\documentclass[12pt,a4paper]{amsart}

\usepackage{amssymb}
\usepackage{amscd}

\def\frak{\mathfrak}
\def\Bbb{\mathbb}
\def\Cal{\mathcal}

\let\phi\varphi

\newcommand{\x}{\times}
\renewcommand{\o}{\circ}

\newcommand{\al}{\alpha}
\newcommand{\be}{\beta}
\newcommand{\ga}{\gamma}

\newcommand{\ka}{\kappa}

\newcommand{\om}{\omega}
\newcommand{\ph}{\phi}
\newcommand{\ps}{\psi}
\renewcommand{\th}{\theta}
\newcommand{\si}{\sigma}
\newcommand{\ze}{\zeta}
\newcommand{\Ga}{\Gamma}
\newcommand{\La}{\Lambda}
\newcommand{\Ph}{\Phi}

\newcommand{\Om}{\Omega}

\newcommand{\im}{\operatorname{im}}

\newcommand{\gr}{\operatorname{gr}}
\newcommand{\id}{\operatorname{id}}

\newcommand{\ad}{\operatorname{ad}}

\newcommand{\Ad}{\operatorname{Ad}}
\newcommand{\Adb}{\operatorname{\underline{Ad}}}
\newcommand{\adb}{\operatorname{\underline{ad}}}

\newcounter{theorem}
\numberwithin{equation}{section}

\newtheorem{thm}[theorem]{Theorem}
\newtheorem*{thm*}{Theorem \thesubsection}
\newtheorem{lemma}[theorem]{Lemma}
\newtheorem{prop}[theorem]{Proposition}
\newtheorem{cor}[theorem]{Corollary}
\newtheorem*{lemma*}{Lemma \thesubsection}
\newtheorem*{prop*}{Proposition \thesubsection}
\newtheorem*{cor*}{Corollary \thesubsection}

\theoremstyle{definition}
\newtheorem{definition}[theorem]{Definition}
\newtheorem*{definition*}{Definition \thesubsection}

\newtheorem*{example*}{Example \thesubsection}
\theoremstyle{remark}
\newtheorem{remark}[theorem]{Remark}
\newtheorem*{remark*}{Remark \thesubsection}

\def\sideremark#1{\ifvmode\leavevmode\fi\vadjust{\vbox to0pt{\vss
 \hbox to 0pt{\hskip\hsize\hskip1em
 \vbox{\hsize3cm\tiny\raggedright\pretolerance10000
  \noindent #1\hfill}\hss}\vbox to8pt{\vfil}\vss}}}%
                        
\begin{document}
\renewcommand{\today}{} 

\title{Partial AHS-Structures, their Cartan description\\ 
and partial BGG sequences}

\author{Andreas \v Cap and Micha\l\ Andrzej Wasilewicz}

\address{Faculty of Mathematics\\
  University of Vienna\\
  Oskar--Morgenstern--Platz 1\\
  1090 Wien\\
  Austria}

\date{April 23, 2025}

\email{Andreas.Cap@univie.ac.at}
\email{michal.wasilewicz@univie.ac.at}



\begin{abstract}
  G-structures and Cartan geometries are two major approaches to the description of
  geometric structures (in the sense of differential geometry) on manifolds of some
  fixed dimension $n$. We show that both descriptions naturally extend to the setting
  of manifolds of dimension $\geq n$ which are endowed with a distinguished
  involutive distribution $F$ of rank $n$. The resulting ``partial'' structures are
  most naturally interpreted as smooth families of standard G-structures or Cartan
  geometries on the leaves of the foliation defined by $F$.

  We prove that for the special class of AHS-structures (also known as $|1|$-graded
  parabolic geometries) the construction of a canonical Cartan geometry associated to
  a G-structure extends to this general setting. As an application, we prove that for
  partial AHS-structures there is an analog of the machinery of BGG sequences. This
  constructs sequences of differential operators of arbitrarily high order intrinsic
  to the structures. Under appropriate flatness conditions, these sequence are fine
  resolutions of sheaves which locally can be realized as pullbacks of sheaves on
  local leaf spaces for the foliation defined by $F$.
  \end{abstract}

  \thanks{This research was funded in whole or in part by the Austrian Science Fund
    (FWF): 10.55776/P33559 and 10.55776/Y963. For open access purposes, the authors
    have applied a CC BY public copyright license to any author-accepted manuscript
    version arising from this submission.  This article is based upon work from COST
    Action CaLISTA CA21109 supported by COST (European Cooperation in Science and
    Technology). https://www.cost.eu. We thank the anonymous referee for several 
    remarks that helped improving the paper.} 

  \maketitle

\pagestyle{myheadings}\markboth{\v Cap and Wasilewicz}{Partial AHS-structures}

\section{Introduction}\label{1}
The concept of G-structures represents the most important approach to a description
of geometric structures in the sense of differential geometry. The idea is to endow
each tangent space of a smooth manifold $M$ of dimension $n$ with some ``structure''
isomorphic to a ``model structure'' on $\Bbb R^n$ (or some vector space of real
dimension $n$) in a way depending smoothly on the base point. This ``structure'' can
be simple, like an inner product or some tensor of fixed type, or more intricate like
a subvariety in the tangent space or its projectivization, or an identification with
a tensor product of spaces of fixed dimensions. There is a uniform concept based on a
closed subgroup $G_0\subset GL(n,\Bbb R)$, which can be thought of as the group of
linear automorphisms of the model structure. The formal definition of a
$G_0$-structure (or a G-structure with structure group $G_0$) is then via a reduction
of structure group of the (linear) frame bundle of $M$. To include analogs of spin
structures it is preferable to allow $G_0$ to be a covering of a closed subgroup of
$GL(n,\Bbb R)$, which is easily possible. 

The geometric properties of such structures can be very diverse and the general
theory of G-structures relates them to algebraic properties of the structure group
$G_0$. The main tool here is the study of connections that are compatible with the
given structure and in particular of their torsion. This leads to results on
existence of canonical connections (like for Riemannian metrics, which are equivalent
to $O(n)$-structures) or of families of distinguished connections.

Cartan geometries represent a different general approach to a concept of geometric
structures. The basic idea here is to view a manifold $M$ endowed with a geometric
structure as a ``curved analog'' of a homogeneous space $G/P$. For Cartan geometries,
many properties of general geometries are more closely related to those of the
homogeneous model $G/P$. However, for many choices of $P\subset G$, the resulting
structures are not as interesting from a geometric point of view, since they may
involve data like the choice of a compatible connection which heavily influences
geometric properties.

There are overlaps between the two concepts, which represent many very interesting
geometric structures. A basic example is that $G_0$-structures that determine
a canonical connection can always be equivalently described as Cartan
geometries. While the Cartan point does not provide additional technical input in
these cases, it leads to a very nice and conceptual picture (see \cite{Sharpe}) which
generalizes the point of view that Riemannian manifolds are curved analogs of
Euclidean space.

There also are examples of $G_0$-structures that do not admit canonical connections
but still admit an equivalent description via a Cartan geometry. The most important
examples of this situation are AHS-structures and among them conformal
pseudo-Riemannian structures. Here the group $P$ is an extension of $G_0$, but the
group $G$ is a new ingredient, which is not easy to guess knowing only $G_0$. In the
conformal Riemannian case, $G_0=CO(n)$ the conformal group in dimension $n$ and $P$
is a semi-direct product of $G_0$ with the vector group $\Bbb R^{n*}$. However, $G$
is the (semi-)simple group $SO(n+1,1)$ of conformal isometries of the sphere
$S^n$. The tools derived from the Cartan description, like tractor bundles and BGG
sequences have lead to many important developments in conformal geometry and related
areas during the last decades.

\medskip

The aim of this article is to generalize both approaches and the relations between
them to a situation of smooth families of structures. An obvious idea would be to
consider families of structures on the fibers of a locally trivial bundle or, more
generally, of a surjective submersion. It turns out, however, that one can easily
deal with the even more general setting of involutive distributions (of constant
rank) on arbitrary manifolds. The resulting geometries can be interpreted as smooth
families of structures on the leaves of the induced foliation. Locally, they are
equivalent to families on fiber bundles, but, depending on global properties of the
foliation, interesting global phenomena may arise. Using the right approach, the
generalization is relatively straightforward, so we include some background on
G-structures and Cartan geometries in order to clarify this approach. Since there are
already quite a few concepts of geometric structures related to foliations around and
since technically we replace differential forms by partially defined differential
forms, we call the resulting concepts \textit{partial G-structures} and
\textit{partial Cartan geometries}.

The main result on these concepts we prove is that one can canonically associate to
any partial AHS-structure a normal partial Cartan geometry. To do this, one can use
the same normalization condition and all the algebraic results available for standard
AHS-structures. Moreover, there is a functorial geometric construction of the partial
Cartan geometry phrased in terms of distinguished partial connections compatible with a
partial AHS-structure (with normalized torsion). This construction is given in the
proof of Theorem \ref{thm2.7}, which shows that any partial AHS-structure is induced
by some normal partial Cartan geometry. Theorem \ref{thm2.8} is a strong uniqueness
result, showing in particular that a normal partial Cartan geometry is uniquely
determined up to isomorphism by the underlying partial AHS-structure.

In Section \ref{3} of the article, we give an application of the Cartan description
of partial AHS-structures. We show that there is an analog of the machinery of BGG
sequences, as originally introduced in \cite{CSS-BGG}, in the partial setting. In
particular, this shows that for any type of partial AHS-structure, there are
differential operators of arbitrarily high order intrinsically associated to the
structure, see \S \ref{3.6}. The machinery provides an extension of and a new
perspective on the relative BGG machinery introduced in \cite{Rel-BGG1} and
\cite{Rel-BGG2}. To some extent, the point of view we obtain is closer to the way
relative BGG resolutions were first used in twistor theory, see in particular the
book \cite{BEastwood}. There the authors considered fibrations between generalized flag
manifolds of the same group, observed that then the fibers are generalized flag
manifolds of a smaller group and used a ``BGG resolution along the fibers''. The
point of view of geometric structures on the fibers played only a small role in the
geometric version of the relative BGG machinery and, apart from one case discussed in
the thesis \cite{Michal-thesis} of the second author, has not been studied in detail
so far. The partial theory we develop actually shows that for the simplest examples
of structures, a ``BGG machinery along the fibers'' is available in a much more
general setting. We also show that if the partial Cartan geometry has vanishing
curvature, then partial BGG sequences are complexes and resolutions of sheaves that
can be locally described as pullbacks of sheaves on local leaf spaces of the
foliation, see Theorem \ref{thm3.6}. The same theorem also provides an analog of the
concept of normal BGG solutions in the partial setting.

\section{Partial AHS-structures and their Cartan description}\label{2}

\subsection{$G_0$-structures and AHS-structures}\label{2.1}
A closed subgroup $G_0\subset GL(n,\Bbb R)$ gives rise to a geometric structure on
manifolds of dimension $n$. Such a $G_0$-structure is defined as a reduction of the
linear frame bundle $\Cal PM\to M$ of a smooth $n$-manifold $M$ to the structure
group $G_0\subset GL(n,\Bbb R)$. Slightly more generally, one can consider the case
that $G_0$ is a covering of a closed subgroup $\underline{G_0}\subset GL(n,\Bbb R)$
to include structures like spin structures. In this case one has to use the
definition as a principal $G_0$-bundle $p:\Cal G_0\to M$ together with a morphism
$\Cal G_0\to \Cal PM$ of fiber bundles, which is equivariant over the obvious
homomorphism $G_0\to GL(n,\Bbb R)$ coming from the covering homomorphism
$G_0\to\underline{G}_0$.

It is well known that such a homomorphism can be equivalently encoded as a $\Bbb
R^n$-valued one-form $\th\in\Om^1(\Cal G_0,\Bbb R^n)$, which is strictly horizontal
and equivariant for the natural representation of $G_0$ on $\Bbb R^n$ coming from the
homomorphism $G_0\to GL(n,\Bbb R)$ from above. Here being strictly horizontal means
that for each $u\in\Cal G_0$ the kernel of $\th(u):T_u\Cal G_0\to\Bbb R^n$
coincides with the vertical subspace $\ker(T_up)$. Since $M$ is assumed to be
$n$-dimensional, this means that $\th(u)$ descends to a linear isomorphism
$T_{p(u)}M\to\Bbb R^n$ which defines the homomorphism to the frame bundle.

AHS-structures are a special class of such structures that arise from simple Lie
algebras. To specify such a structure, one starts from a Lie group $G$ whose Lie
algebra $\frak g$ is simple and endowed with a so-called $|1|$-grading, i.e.\ a
vector space decomposition $\frak g=\frak g_{-1}\oplus\frak g_0\oplus\frak g_1$ such
that $[\frak g_i,\frak g_j]\subset\frak g_{i+j}$ with the convention that $\frak
g_{\ell}=\{0\}$ if $|\ell|>1$. Such gradings are well understood from the
classification of Hermitian symmetric spaces, see Sections 3.2.3 and 3.2.10 of
\cite{book} for more details and the classification. The subalgebra $\frak p:=\frak
g_0\oplus\frak g_1$ is then a parabolic subalgebra of $\frak g$. It is well known
that the normalizer $\{g\in G:\Ad(g)(\frak p)\subset\frak p\}$ is a closed subgroup
of $G$ that has Lie algebra $\frak p$. One then chooses a \textit{parabolic subgroup}
$P\subset G$ corresponding to $\frak p$, i.e.\ a subgroup that lies between this
normalizer and its connected component of the identity. In particular, this implies
that $P$ has Lie algebra $\frak p$.

It is also well known that the exponential mapping restricts to a diffeomorphism from
$\frak g_1$ onto a closed normal subgroup $P_+\subset P$, so $P/P_+$ has Lie algebra
$\frak p/\frak g_1\cong\frak g_0$. It turns out that
$G_0:=\{g\in P:\forall i:\Ad(g)(\frak g_i)\subset\frak g_i\}$, which obviously is a
closed subgroup of $P$, projects isomorphically onto $P/P_+$, so $P$ is the
semi-direct product of $G_0$ and $P_+$. By construction, the group $G_0$ acts via the
adjoint action on each $\frak g_i$ and one shows that the resulting map
$G_0\to GL(\frak g_{-1})$ is infinitesimally injective, so putting
$n:=\dim(\frak g_{-1})$ there is the concept of a $G_0$-structure on manifolds of
dimension $n$. Of course, one can view the equivariant form $\th$ describing such a
structure as having values in $\frak g_{-1}\cong\frak g/\frak p$.

Some choices of $G$ and $P$ show exceptional behavior from the point of
  view of the resulting $G_0$-structures and we have to exclude them in what
  follows. There is a $|1|$-grading of $\frak{sl}(n+1,\Bbb K)$ with $\frak
  g_{-1}=\Bbb K^n$ and $\frak g_0\cong\frak{gl}(n,\Bbb K)$ for $\Bbb K=\Bbb R$ or
  $\Bbb C$ and this is the only available $|1|$-grading if $n=1$. Hence a
  $G_0$-structure in this case contains almost no information. For $n\geq 2$, the
  corresponding pairs $(G,P)$ are of interest for the picture of Cartan geometries
  that we will discuss below, since they lead to equivalent descriptions of classical
  projective structures and an almost complex analog known as $c$-projective
  structures, see \cite{c-proj}. Since these structures need a different description,
  we will not consider them in this article. The case of projective structures is
  dealt with in in detail in \cite{Michal-thesis} and we'll only make short remarks
  on this case below. For clarity of terminology, we phrase this as a definition.
 
  \begin{definition}\label{def2.1}
    An \textit{AHS-pair} $(G,P)$ consist of a simple Lie group $G$ and a parabolic
    subgroup $P\subset G$ corresponding to a $|1|$-grading of the Lie algebra
    $\frak g$ of $G$, which is different from the grading of $\frak{sl}(n+1,\Bbb K)$
    with $\frak g_{-1}=\Bbb K^n$ and $\frak g_0\cong\frak{gl}(n,\Bbb K)$ for
    $\Bbb K=\Bbb R$ or $\Bbb C$. The subgroup $G_0\subset P$ of all elements whose
    adjoint action preserves the grading of $\frak g$ will be called the structure
    group of the AHS-structure determined by $(G,P)$.
  \end{definition}

The importance of AHS-structures as defined here is visible from the classification
results of S.\ Kobayashi and T.\ Nagano in \cite{KN1}, which use the concept of
prolongation of $G_0$-structures. Equivalently, this can be phrased in terms of the
question whether automorphisms of the structure in question are locally determined by
some finite jet in one point. The classical example of this situation is provided by
pseudo-Riemannian metrics, since local isometries for such metrics are determined
uniquely in a neighborhood of a point by their one-jet in that point. On the other
end of the spectrum, there are (almost) symplectic structures, for which
automorphisms are not locally determined by any finite jet in a point. The main
result of \cite{KN1} then is that the AHS-structures are the only $G_0$-structures
for which $G_0$ acts irreducibly on $\frak g_{-1}$ and automorphisms are locally
determined by some finite jet in a point but not by the one-jet. Indeed,
automorphisms for any AHS-structure are always locally determined by their two-jet in
a point, which makes them difficult to deal with. The best known example of an
AHS-structure are conformal structures in dimensions $\geq 3$, see \cite{AHS2} for
more examples of AHS-structures.

\subsection{Partial $G_0$-structures}\label{2.2}
An obvious idea for defining smooth families of geometric structures is to look at
fiber bundles and geometric structures on each fiber which, in an appropriate sense,
depend smoothly on the base point. It turns out that for $G_0$-structures (and also
for Cartan geometries as discussed below), there is a natural generalization, which
can locally be interpreted as a smooth family in that sense. Consider a smooth
manifold $M$ of dimension $n+k$ endowed with a distinguished distribution
$F\subset TM$ of rank $n$. Then consider the linear frame bundle
$\Cal PF=GL(\Bbb R^n,F)$ of the vector bundle $F$. This is a principal fiber bundle
with structure group $GL(n,\Bbb R)$ whose fiber over a point $x\in M$ consists of all
linear isomorphisms $\Bbb R^n\to F_x$, the fiber of $F$ over $x$. In principle, one
could define partial $G_0$-structures for an arbitrary distribution $F$, but as far
as we know interesting results are only available in the case that $F$ is involutive,
so we restrict to that case.

\begin{definition}\label{def2.2.1}
  Let $G_0$ be a covering of a closed subgroup $\underline{G}_0\subset GL(n,\Bbb R)$,
  $M$ a smooth manifold of dimension $\geq n$, and $F\subset TM$ an involutive
  distribution of rank $n$. Then a \textit{partial $G_0$-structure on $(M,F)$} is
  given by a principal fiber bundle $\Cal G_0\to M$ with structure group $G_0$, and a
  morphism $\Ph:\Cal G_0\to\Cal PF$ of fiber bundles with base map $\id_M$, which is
  equivariant over the homomorphism $G_0\to GL(n,\Bbb R)$ induced by the covering
  homomorphism $G_0\to\underline{G}_0$. If the dimension of $M$ is $n+k$, then we
  will call this a \textit{partial $G_0$-structure of codimension $k$}.
\end{definition}

To define morphisms between partial $G_0$-structures, we need some
observations. Consider two manifolds $M^i$ endowed with involutive rank $n$
distributions $F^i\subset TM^i$ for $i=1,2$ and let $f:M^1\to M^2$ be a smooth
map. Suppose that for each $x\in M^1$, the tangent map $T_xf:T_xM^1\to T_{f(x)}M^2$
restricts to a linear isomorphism $F^1_x\to F^2_{f(x)}$. Then there is natural lift
of $f$ to a principal bundle homomorphism $\Cal Pf:\Cal PF^1\to\Cal PF^2$ defined by
$\Cal Pf(u)=T_xf\circ u:\Bbb R^n\to F^2_{f(x)}$ for $u\in\Cal P_xF^1$, i.e.\
$u:\Bbb R^n\to F^1_x$. Using this, we can now introduce the natural notion of
morphisms.

\begin{definition}\label{def2.2.2}
  Let $(\Cal G_0^i\to M^i,F^i,\Phi^i)$ be partial $G_0$-structures as in Definition
  \ref{def2.2.1} for $i=1,2$. The a \textit{morphism}
  $(\Cal G_0^1\to M^1,F^1,\Phi^1)\to (\Cal G_0^2\to M^2,F^2,\Phi^2)$ is a
  homomorphism $\Psi:\Cal G_0^1\to\Cal G_0^2$ of principal bundles such that the base
  map $\psi:M^1\to M^2$ has the property that for any $x\in M^1$, the tangent map
  $T_x\psi$ restricts to a linear isomorphism $F^1_x\to F^2_{\psi(x)}$ and such that
  $\Phi^2\o \Psi=\Cal P\psi\o\Phi^1$.
\end{definition}

Observe that in the case that $\dim(M)=n$ and $F=TM$, we recover the usual concepts
of $G_0$-structures and their morphisms (which then automatically are local
diffeomorphisms). Similarly as in that setting, partial $G_0$-structures can be
described via an analog of strictly horizontal, equivariant forms, but we need some
background to introduce this. Let $p:\Cal G_0\to M$ be a principal fiber bundle. Then
for each $u\in\Cal G_0$ with $x=p(u)$ the tangent map $T_up:T_u\Cal G_0\to T_xM$ is
surjective, and we define $\widetilde{F}_u:=(T_up)^{-1}(F_x)\subset T_u\Cal
G_0$. Evidently, these spaces fit together to define a smooth distribution
$\widetilde{F}\subset T\Cal G_0$ which contains the vertical subbundle $V\Cal G_0$ of
$\Cal G_0\to M$. In particular, given a vector space $E$, we can consider the vector
bundle $L(\widetilde{F},E)\to\Cal G_0$, whose fiber over a point $u\in\Cal G_0$ is
the space $L(\widetilde{F}_u,E)$ of linear maps from $\widetilde{F}_u$ to
$E$. Sections of this vector bundle can be considered as ``partially defined
$E$-valued differential forms'' on $\Cal G_0$.

For a section $\th\in\Ga(L(\widetilde{F},\Bbb R^n))$, there is an obvious notion of
being \textit{strictly horizontal}, namely that for each $u\in\Cal G_0$, the kernel
of $\th(u):\widetilde{F}_u\to\Bbb R^n$ coincides with the vertical subspace $V_u\Cal
G_0=\ker(T_up)$. Moreover, for $g\in G_0$ the principal right action $r^g:\Cal
G_0\to\Cal G_0$ satisfies $p\o r^g=p$, which readily implies that $\widetilde{F}$ is
invariant under the principal right action in the sense that $T_ur^g:T_u\Cal G_0\to
T_{u\cdot g}\Cal G_0$ maps $\widetilde{F}_u$ to $\widetilde{F}_{u\cdot g}$. This
means that given $\al\in\Ga(L(\widetilde{F},E))$ for some vector space $E$ and $g\in
G_0$, there is a well defined pull back $(r^g)^*\al\in\Ga(L(\widetilde{F},E))$
defined by the usual formula $((r^g)^*\al)(u)(\xi)=\al(u\cdot g)(T_ur^g(\xi))$. Given
a representation of $G_0$ on $E$, one can thus define equivariancy of
$\al\in\Ga(L(\widetilde{F},E))$ in the usual way. Using these observations, we can
now formulate

\begin{prop}\label{prop2.2}
Let $G_0$ be a covering of a closed subgroup $\underline{G}_0\subset GL(n,\Bbb R)$
and let $K$ be the kernel of the covering homomorphism $G_0\to\underline{G}_0$.
  
(1) For a manifold $M$ of dimension $\geq n$ and an involutive distribution
$F\subset TM$, a partial $G_0$-structure on $(M,F)$ is equivalent to a principal
$G_0$-bundle $p:\Cal G_0\to M$ endowed with a strictly horizontal, $G_0$-equivariant
section $\th\in\Ga(L(\widetilde{F},\Bbb R^n))$.

(2) For two partial $G_0$-structures described as in (1), suppose that we have given
a principal bundle morphism $\Psi:\Cal G_0^1\to\Cal G_0^2$, whose base map
$\psi:M^1\to M^2$ satisfies the condition from Definition \ref{def2.2.2}. Then there
is a well defined pullback $\Psi^*\th^2\in\Ga(L(\widetilde{F}^1,\Bbb R^n)$ and $\Psi$
is a morphisms of partial $G_0$-structures, if and only if $\Psi^*\th^2=\th^1$.

(3) A morphism $\Psi:\Cal G_0^1\to\Cal G_0^2$ of partial $G_0$-structures is
determined by its base map $\psi$ up to a locally constant function $\Cal G_0^1\to
K$.
\end{prop}
\begin{proof}
  These are simple variations of the standard proofs for $G_0$-structures.

  (1) We have to show that a morphism $\Ph:\Cal G_0\to \Cal PF$ as in Definition
  \ref{def2.2.1} is equivalent to a section $\th$ with the specified properties. One
  first observes that on $\pi:\Cal PF\to M$, there is a tautological section
  $\underline{\th}$ of $L(\widetilde{F},\Bbb R^n)$, which is strictly horizontal and
  $GL(n,\Bbb R)$ equivariant. This is defined by
  $\underline{\th}(u)(\xi):=u^{-1}(T_u\pi(\xi))$, taking into account that
  $u:\Bbb R^n\to F_{\pi(u)}$ is a linear isomorphism and $T_u\pi(\xi)\in F_{\pi(u)}$
  for $\xi\in\widetilde{F}_u$. Given a morphism $\Phi$ as above, one observes that
  the tangent maps of $\Phi$ send the pre-image of $F$ in $T\Cal G_0$ to the
  pre-image of $F$ in $T\Cal PF$. Hence we can define $\th:=\Phi^*\underline{\th}$
  and verify that this has the required properties.

Conversely, given $(p:\Cal G_0\to M,\th)$ we observe that for $u\in\Cal G_0$, $T_up$
restricts to a map $\widetilde{F}_u\to F_{p(u)}$ which descends to a linear
isomorphism $\widetilde{F}_u/V_u\Cal G_0\to F_{p(u)}$. On the other hand, since
$\ker(\th(u))=V_u\Cal G_0$ it descends to an injective linear map
$\widetilde{F}_u/V_u\Cal G_0\to\Bbb R^n$ which has to be a linear isomorphism since
both spaces have the same dimension. Composing the inverse of this isomorphism with
the above isomorphism, one obtains a linear isomorphism $\Phi(u):\Bbb R^n\to F_u$, and
one directly verifies that these fit together to define a morphism $\Phi:\Cal
G_0\to\Cal PF$ with the required properties.

(2) The base map $\psi$ of $\Psi$ has the property that for each $x\in M^1$, the
tangent map $T_x\psi$ restricts to a linear isomorphism $F^1_x\to F^2_{\psi(x)}$ and,
as discussed above, this implies that there is an induced principal bundle homomorphism
$\Cal P\psi:\Cal PF^1\to\Cal PF^2$. By definition, we have $Tp^2\o T\Psi=T\psi\o
Tp^1$, so for each $u\in\Cal G_0^1$, we get
$T_u\Psi(\widetilde{F}^1_u)\subset\widetilde{F}^2_u$ and an element in the kernel of
this map has to lie in the kernel of $T_up^1$. This shows that the restriction of
$T_u\Psi$ to $\widetilde{F}^1_u$ maps this space injectively and hence bijectively to
$\widetilde{F}^2_u$. This readily implies that $\Psi^*\th^2$ is well
defined.

For $\xi\in\widetilde{F}^1_u$ and $x=p^1(u)$, we get
$T_up^1(\xi)=\Phi^1(u)(\th^1(u)(\xi))$ and hence
$$\Cal P\psi(\Phi^1(u))(\th^1(u)(\xi))=T_x\psi(T_up^1(\xi)).$$
On the other hand, $\Psi^*\th^2(u)(\xi)=\th^2(\Psi(u))(T_u\Psi(\xi))$, and
$\Phi^2(\Psi(u))$ maps this element to
$T_{\Psi(u)}p^2(T_u\Psi(\xi))=T_x\psi(T_up^1(\xi))$. This says that as maps
$\widetilde{F}^1_u\to F^2_{\psi(x)}$, we get
$$
\Cal P\ps\o \Phi^1\o \th^1(u)=\Phi^2\o\Psi\o(\Psi^*\th^2)(u).
$$
Since both sides are surjective and have kernel the vertical subspace, we conclude
that $\Cal P\ps(\Phi^1(u))=\Phi^2(\Psi(u))$ is equivalent to
$\th^1(u)=(\Psi^*\th^2)(u)$, so (2) follows. 

(3) Let $\Psi,\widetilde{\Psi}:\Cal G_0^1\to\Cal G_0^2$ be two morphisms with the
same base map $\psi:M^1\to M^2$. Then there is a smooth function $g:\Cal G_0^1\to
G_0$ such that $\widetilde{\Psi}(u)=\Psi(u)\cdot g(u)$ for any $u\in\Cal G_0^1$. But
since both maps are morphisms, we know that $\Phi^2\o\widetilde{\Psi}=\Phi^2\o\Psi$,
which implies that $g(u)\in K$ for any $u$. But by assumption $G_0$ is a 
covering of $\underline{G}_0$, so $K$ is discrete and hence the smooth function $g$
has to be locally constant. 
\end{proof}

In particular, given an AHS-pair $(G,P)$ and the structure group $G_0$ of the
corresponding AHS-structure as in Definition \ref{def2.1}, we have the notion of a
partial AHS-structure on a manifold of dimension $\geq n:=\dim(\frak g_{-1})$ endowed
with an involutive distribution of rank $n$. This can be described as a principal
$G_0$-bundle $p:\Cal G_0\to M$ endowed with a strictly horizontal, $G_0$-equivariant
section $\th\in\Ga(L(\widetilde{F},\frak g_{-1}))$.

\subsection{Examples}\label{2.3}
(1) Let us first discuss how smooth families of $G_0$-structures fit into the picture
of partial $G_0$-structures. Suppose that $f:M\to N$ is a surjective submersion with
$\dim(N)=k$ and $\dim(M)=n+k$ and let $F\subset TM$ be the vertical subbundle of $f$,
i.e.\ $F_x=\ker(T_xf:T_xM\to T_{f(x)}N)$. Then it is well known that for each $y\in
N$, the fiber $M_y:=f^{-1}(\{y\})\subset M$ is a smooth submanifold, so we can
naturally view $M$ as a smooth family of $n$-dimensional manifolds parametrized by
the $k$-dimensional manifold $N$. Moreover, for $x\in M$ with $f(x)=y$, the vertical
subspace $F_x$ is exactly the tangent space $T_xM_y$ to the fiber. Thus the
restriction $\Cal PF|_{M_y}$ can be naturally identified with the linear frame bundle
of $M_y$. For a partial $G_0$-structure $(\Cal G_0\to M,\th)$ we can therefore view
the restriction $\Cal G_0|_{M_y}$ endowed with the restriction of $\th$ as a
$G_0$-structure on $M_y$. Hence in this situation a partial $G_0$-structure of
codimension $k$ can be naturally viewed as a $k$-dimensional smooth family of
$G_0$-structures.

For a general involutive distribution $F\subset TM$, this picture is available
locally. By the Frobenius theorem there are local leaf spaces, i.e.\ each $x\in M$
has a neighborhood $U\subset M$ for which there is a surjective submersion $f:U\to N$
onto a smooth manifold $N$ such that $\ker(T_zf)=F_z$ for any $z\in U$. The global
picture is more subtle, the best way to think about this is as a smooth family of
$G_0$-structures on the leaves of the foliation defined by $F$. Consider such a leaf,
i.e.\ an injective immersion $i:L\to M$ whose image is a maximal connected integral
submanifold for $F$.  Given a partial $G_0$-structure $(\Cal G_0\to M,\th)$, we can
form the pullback $i^*\Cal G_0\to L$ and we can also define $i^*\th$ by the usual
formula. Now by definition, for each $x\in L$, the tangent map $T_xi$ restricts to a
linear isomorphism $T_xL\to F_x$. Thus $i^*\th$ is defined on all of $T_xL$, so $\th$
is a strictly horizontal $\Bbb R^n$-valued one-form on $i^*\Cal G_0\to L$, making it
into a $G_0$-structure in the standard sense. In this point of view, the smoothness
of the family of $G_0$-structures is a much more subtle condition, however. Recall
that there are well known examples of foliations that have a dense leaf, and for such
a foliation, the whole partial $G_0$-structure is uniquely determined by the induced
structure on that leaf.

(2) Let us discuss explicit descriptions of partial $G_0$-structures in the case of
the most popular AHS-structure, namely conformal structures. The relevant
$|1|$-grading here has the form
$$\frak{so}(n+1,1)=\mathbb R^n\oplus \frak{co}(n)\oplus\mathbb R^{n*},$$ see Section
1.6.3 of \cite{book} for details. Choosing an appropriate group $G$, the adjoint
action on $\frak g_{-1}\cong\Bbb R^n$ identifies the group $G_0$ with the conformal
group $CO(n)$. Now let $M$ be a smooth manifold of dimension $n+k$, $F\subset TM$ an
involutive distribution of rank $n$ and $\Cal PF\to M$ the linear frame bundle of
$F$. Completely parallel to standard $G$-structures, a reduction of $\Cal PF$ to the
structure group $CO(n)\subset GL(n,\Bbb R)$ is equivalent to the choice of a class
$[g_x]$ of non-degenerate positive-definite bilinear forms $g_x$ on the fibers $F_x$
of $F$ which are related to each other by multiplication by a positive number. These
have to depend smoothly on the point $x$ in the sense that locally there are smooth
sections of $S^2F^*$ whose values in each point lie in the chosen class. Otherwise
put, one has to prescribe a smooth ray-subbundle in $S^2F^*$. Since any such bundle
is globally trivial, there are global smooth sections, so one can equivalently define
a partial conformal structure as a equivalence class of smooth positive definite
bundle metrics on $F$, with two such metrics $g$ and $\hat g$ being equivalent if and
only if $\hat g=e^{2f}g$ for some smooth function $f:M\to\Bbb R$.

Here it is also relevant to consider coverings of subgroups of $GL(n,\Bbb
R)$. Indeed, choosing $G$ to be a spin group, one arrives at $G_0=CSpin(n)$ and hence
at the concept of a partial conformal spin structure.

\subsection{Cartan description of AHS-structures}\label{2.4}
The most effective tools for the study of AHS-structures, which take the higher order
phenomena mentioned in Definition \ref{def2.1} into account, are derived from the
equivalent description as Cartan geometries. This makes AHS-structures into the
simplest examples of the class of parabolic geometries that is discussed in the book
\cite{book}. Indeed, as a special case of the results of existences of canonical
Cartan connections in the parabolic case, one obtains canonical Cartan connections
associated to all AHS-structures.

For an AHS-pair $(G,P)$ as in Definition \ref{def2.1}, a Cartan geometry of type $(G,P)$ on a
smooth manifold $M$ of dimension $\dim(G/P)$ is given by a principal $P$-bundle
$p:\Cal G\to M$ endowed with a Cartan connection $\om\in\Om^1(\Cal G,\frak g)$. This
means that for each $u\in\Cal G$, $\om(u):T_u\Cal G\to\frak g$ is a linear
isomorphism, $\om$ is $P$-equivariant, and for the fundamental vector field
$\ze_A\in\frak X(\Cal G)$ generated by $A\in\frak p$, one has $\om(\ze_A(u))=A$ for
any $u\in\Cal G$. There is an obvious associated concept of morphisms: Given
$(p^i:\Cal G^i\to M^i,\om^i)$ for $i=1,2$, a morphism is a principal bundle
homomorphism $\Phi:\Cal G^1\to\Cal G^2$ such that $\Phi^*\om^2=\om^1$ (which
immediately implies that $\Phi$ is a local diffeomorphism). Similarly as in
Proposition \ref{prop2.2}, one proves that under an assumption on infinitesimal
effectivity, which is satisfied in all cases we consider, a morphism of Cartan
geometries is determined by its base map up to a locally constant function, see
Proposition 1.5.3 of \cite{book}.
 
Now it is easy to see that any such Cartan geometry has an underlying AHS-structure
corresponding to $G_0=P/P_+$. Via the principal right action, $P_+\subset P$ acts
freely on $\Cal G$ and the quotient $\Cal G_0:=\Cal G/P_+$ is easily seen to be a
principal fiber bundle over $M$ with structure group $P/P_+=G_0$. Moreover,
projecting the values of $\om$ to $\frak g/\frak p\cong\frak g_{-1}$ one obtains an
equivariant element of $\Om^1(\Cal G,\frak g_{-1})$, which is easily seen to descend
to a strictly horizontal, equivariant form $\th\in\Om^1(\Cal G_0,\frak g_{-1})$. The
underlying AHS-structure then is $(\Cal G_0,\th)$. This construction immediately
implies that a morphism of Cartan geometries induces a morphism between the underling
AHS-structures.

The question is whether one can revert this procedure, i.e.\ associate to a
$G_0$-structure canonically a Cartan geometry which has that underlying
structure. Making choices, it is rather easy to construct some Cartan geometry which
has this property, but doing this in a canonical way turns out to be more
subtle. Indeed one can trivially modify a Cartan connection $\om$ without changing
the underlying structure, so to get a canonical Cartan geometry one has to impose
some \textit{normalization condition} on the Cartan connection. The best way to do
this is via the curvature $K\in\Om^2(\Cal G,\frak g)$ of $\om$ that is defined by
\begin{equation}\label{eq:curature}
K(\xi,\eta)=d\om(\xi,\eta)+[\om(\xi),\om(\eta)] \quad \text{\ for\ } \xi,\eta\in\frak
X(\Cal G). 
\end{equation}
So $K$ measures the extent to which $\om$ fails to satisfy the Maurer-Cartan
equation. This also leads to the fundamental result that $K$ vanishes if and only if
$(\Cal G\to M,\om)$ is locally isomorphic (as a Cartan geometry) to $G\to G/P$ with
the Maurer-Cartan form of $G$ as its Cartan connection. It is easy to see that $K$ is
$P$-equivariant and horizontal, i.e.\ it vanishes upon insertion of one vertical
vector field.

Via the Cartan connection $\omega$, the curvature $K$ can be equivalently encoded in
the \textit{curvature function} $\ka:\Cal G\to \La^2\frak g^*\otimes\frak g$ which is
characterized by $K(u)(\xi,\eta)=\ka(u)(\om(u)(\xi),\om(u)(\eta))$. The properties of
$K$ imply that $\ka$ is a $P$-equivariant function and actually has values in
$\La^2(\frak g/\frak p)^*\otimes\frak g$. Via the Killing form $(\frak g/\frak
p)^*\cong\frak g_1$, so we can also view the values of $\ka$ as lying in $\La^2\frak
g_1\otimes\frak g$. Now $\frak g_1\subset\frak g$ is an abelian Lie subalgebra which
acts on $\frak g$ by the restriction of the adjoint action, so there is a Lie algebra
homology differential, which for historical reasons is usually called the
\textit{Kostant codifferential} and denoted by $\partial^*:\La^2\frak g_1\otimes\frak
g\to\frak g_1\otimes\frak g$. The construction implies that $\partial^*$ is
$P$-equivariant, the explicit formula is stated in formula \eqref{codiff} below. A
Cartan geometry of type $(G,P)$ is called \textit{normal} if and only if
$\partial^*\o\ka=0$. Now we can state the fundamental theorem on existence and
uniqueness of canonical Cartan connections for AHS-structures, which is proved in
Section 3.1 of \cite{book} in the more general setting of parabolic geometries:

\begin{thm}\label{thm2.4}
  Consider an AHS-pair $(G,P)$ as in Definition \ref{def2.1} and let $G_0\subset P$
  be the structure group of the corresponding AHS-structure. For any $G_0$-structure
  $(\Cal G_0\to M,\th)$, there is a normal Cartan geometry $(p:\Cal G\to M,\om)$ with
  this underlying $G_0$-structure, so in particular $p$ factorizes through
  $\Cal G_0$. Moreover, if $(p^i:\Cal G^i\to M^i)$ are normal Cartan geometries for
  $i=1,2$ and $\Phi:\Cal G^1_0\to\Cal G^2_0$ is a morphism of the underlying
  $G_0$-structures, then there is a unique lift $\tilde\Phi:\Cal G^1\to\Cal G^2$ of
  $\Phi$ which is a morphism of Cartan geometries. In particular, two normal Cartan
  geometries inducing the same underlying $G_0$-structure $(\Cal G_0\to M,\th)$ are
  isomorphic via an isomorphism that induces the identity on $\Cal G_0$.
\end{thm}

It is this Cartan description that we want to extend to partial $G_0$-structures.

\subsection{Partial Cartan geometries}\label{2.5}
Returning to the setting of \S \ref{2.2}, consider a manifold $M$ of dimension $k+n$
endowed with an involutive distribution $F\subset TM$ of rank $n$, a Lie group $G$
and a closed subgroup $P\subset G$ such that $\dim(G/P)=n$. Then for a principal
fiber bundle $p:\Cal G\to M$ with structure group $P$ we have the subbundle
$\widetilde F\subset T\Cal G$ from \S \ref{2.2}, so $\xi\in \widetilde{F}_u$ if and only
if $T_up(\xi)\in F_{p(u)}\subset T_{p(u)}M$. We have already observed there that
this subbundle is preserved by the principal right action, i.e.\ for any $g\in P$ and
$u\in\Cal G$, we have $T_ur^g(\widetilde F_u)=\widetilde F_{u\cdot g}$. Consequently,
for any partially defined one form $\tau\in\Ga(L(\widetilde F,V))$ with values in a
vector space $V$ and $g\in P$, we can form $(r^g)^*\tau\in \Ga(L(\widetilde
F,V))$. Having made these observations, there is an obvious partial analog of the
definition of Cartan geometries:

\begin{definition}\label{def2.5}
  A \textit{partial Cartan geometry} of type $(G,P)$ on $(M,F)$ is given by a
  principal fiber bundle $p:\Cal G\to M$ with structure group $P$ and a
  \textit{partial Cartan connection} $\om\in\Gamma(L(\widetilde F,\frak g))$ such
  that
\begin{itemize}
\item[(i)] For any $g\in P$, we have $(r^g)^*\om=\Ad(g^{-1})\o\omega$. 
\item[(ii)] For the fundamental vector field $\ze_A\in\frak X(\Cal G)$ generated by
  $A\in\frak p$, we get $\omega(\ze_A)=A$. 
  \item[(iii)] For any point $u\in\Cal G$, $\om(u):\widetilde F_u\to\frak g$ is a
    linear isomorphism.
\end{itemize}
\end{definition}

There also is a clear way how to define a morphism between two partial Cartan
geometries $(p^i:\Cal G^i\to M^i,\om^i)$ of type $(G,P)$ on $(M^i,F^i)$ for
$i=1,2$. Let $\Phi:\Cal G^1\to\Cal G^2$ be a morphism of principal fiber bundles,
whose base map $\phi$ has the property that for any $x\in M^1$ we get
$T_x\phi(F^1_x)\subset F^2_{\phi(x)}$. Then for any $u\in\Cal G^1$ we get
$T_u\Phi(\widetilde F^1_u)\subset \widetilde F^2_{\Phi(u)}$ and hence we can form
$\Phi^*\om^2\in\Ga(L(\widetilde F^1,\frak g))$. We call $\Phi$ a \textit{morphism} of
partial Cartan geometries if this condition on the base map $\phi$ is satisfied and
$\Phi^*\om^2=\om^1$. Observe that this implies that for any $u\in\Cal G$, the tangent
map $T_u\Phi$ restricts to a linear isomorphism
$\widetilde F^1_u\to\widetilde F^2_{\Phi(u)}$, but it does not even force $M^1$ and
$M^2$ to have the same dimension. As before, we recover the usual concepts of Cartan
geometries and their morphisms in the case that $\dim(M)=n$ and $F=TM$.

\medskip

The general ideas about underlying $G_0$-structures of Cartan geometries extend to
the setting of partial structures introduced here. Given a pair $(G,P)$, we can
consider the natural representation $\Adb:P\to GL(\frak g/\frak p)$ of $P$ on $\frak
g/\frak p$ induced by the restriction of the adjoint representation of $G$. Given a
partial Cartan geometry $(p:\Cal G\to M,\om)$ of type $(G,P)$ on $(M,F)$ we can thus
form the associated bundle $\Cal G\x _P(\frak g/\frak p)$. The defining properties of
the partial Cartan connection then readily imply that the map $\Cal G\x (\frak
g/\frak p)\to F$ that sends $(u,A+\frak p)$ to $T_up(\om(u)^{-1}(A))\in F_{p(u)}$
induces an isomorphism $\Cal G\x _P(\frak g/\frak p)\cong F$.

Under weak assumptions, this can be recast in the setting of \S \ref{2.2}. We just do
this in the setting of AHS-structures and remark on generalizations. So we take and
AHS-pair $(G,P)$ as in Definition \ref{def2.1} and the normal subgroup $P_+\subset P$
such that $P/P_+\cong G_0$. For a partial Cartan geometry $(p:\Cal G\to M,\om)$ on
$(M,F)$, we can then form $\Cal G/P_+$ which is a principal bundle over $M$ with
structure group $P/P_+\cong G_0$. We denote by $p_0:\Cal G_0\to M$ the corresponding
bundle projection.

\begin{prop}\label{prop2.5}
Consider a partial Cartan geometry $(p:\Cal G\to M,\om)$ of type $(G,P)$ over $(M,F)$
and the underlying bundle $p_0:\Cal G_0\to M$ from above and let $\widetilde F\subset
T\Cal G$ and $\widehat F\subset T\Cal G_0$ be the distributions determined by
$F\subset TM$. Then $\om$ descends to a strictly horizontal section
$\th\in\Ga(L(\widehat F,\frak g/\frak p))$ which makes the bundle $\Cal G_0$ into a
partial $G_0$-structure over $(M,F)$. This construction is functorial,
i.e.\ morphisms of partial Cartan geometries descend to morphisms of partial
$G_0$-structures.
\end{prop}
\begin{proof}
Projecting the values of $\om$ to $\frak g/\frak p$, we obtain a smooth section
$\tilde\th$ of the bundle $L(\widetilde F,\frak g/\frak p)\to\Cal G$ which by
construction is $P$-equivariant and vanishes on the vertical subbundle $V\Cal
G=\ker(Tp)$. It is a classical result that the natural projection $q:\Cal G\to\Cal
G_0$ is a principal fiber bundle with structure group $P_+$.  Since for $Z\in\frak
g_1$, the map $\ad(Z)$ maps $\frak g$ to $\frak p$, we conclude that
$\exp(Z)\subset\ker(\Adb)$, so $P_+$ is contained in $\ker(\Adb)$. Hence $\tilde\th$
is $P_+$ invariant and of course it is horizontal with respect to the projection to
$\Cal G_0$. Moreover $T_q$ maps $\widetilde F$ onto $\widehat F$, so we conclude that
$\tilde\th$ descends to a smooth section $\th$ of $L(\widehat F,\frak g/\frak p)\to
G_0$.

The $P$-action on $\widetilde F$ descends to the natural action of $P/P_+=G_0$ on
$\widehat F$ and the representation $\Adb$ also descends to $P/P_+$, so
$P$-equivariancy of $\tilde\th$ implies that $\th$ is $G_0$-equivariant. On the other
hand, for $u\in \Cal G$, $\tilde\th(u)$ is surjective and vanishes on the vertical
subspace $V_u\Cal G$ of $p:\Cal G\to M$. This readily implies that $\th$ is strictly
horizontal and hence makes $p_0:\Cal G_0\to M$ into a partial $G_0$-structure by
Proposition \ref{prop2.2}.

Concerning functoriality, it is clear that a principal bundle morphism $\Phi:\Cal
G^1\to\Cal G^2$ descends to a principal bundle morphism $\underline{\Phi}:\Cal
G^1/P_+\to\Cal G^2/P_+$ with the same base map. If $\Phi^*\om^2=\om^1$ then the
construction above immediately implies that $\underline{\Phi}^*\th^2=\th^1$, so this
is a morphism of partial $G_0$-structures. 
\end{proof}

\begin{remark}\label{rem2.5}
  It is a special feature of the case of AHS-structures that there is a preferred
  subgroup $G_0$ which can serve as the structure group of the underlying
  structure. It follows from the structure theory that in the AHS-case the image
  $\Adb(P)\subset GL(\frak g/\frak p)$ is a closed subgroup $\underline{G}_0$ and
  that $\Adb$ descends to a finite covering $G_0\to\underline{G_0}$.

  For general pairs $(G,P)$, one has to assume that $G_0:=\Adb(P)\subset
GL(\frak g/\frak p)$ is a closed subgroup and define $K:=\ker(\Adb)$. Then one
defines $\Cal G_0:=\Cal G/K$ and proves similarly as in the proof of Proposition
\ref{prop2.5} that one can descend $\om$ to a partially defined one-form with values
in $\frak g/\frak p$ that makes $\Cal G_0$ into a partial $G_0$-structure. Of course,
functoriality of the construction of the underlying $G_0$-structure continues to hold
in this case.
\end{remark}

\subsection{Curvature and normality}\label{2.6}
Let $M$ be a smooth manifold of dimension $n+k$, $F\subset TM$ an involutive
distribution of rank $n$, and $V$ a finite dimensional vector space. Then we refer to
sections of the bundle $L(\La^r F,V)$ as partial $V$-valued $r$-forms and denote the
space of such forms as $\Om^r_F(M,V)$. In this language, the form $\th$ from
Proposition \ref{prop2.2} and a Cartan connection are partial $1$-forms with values
in $\Bbb R^n$ and in $\frak g$, respectively.

Now we first observe that involutivity of $F$ implies that there is a well-defined
exterior derivative $d:\Om^r_F(M,V)\to \Om^{r+1}_F(M,V)$ defined by the usual global
formula. Alternatively, one can extend $\al\in\Om^r_F(M,V)$ to a $V$-valued $r$-form
$\tilde\al$ and then observe that the restriction of $d\tilde\al$ to entries from $F$
depends only on $\al$ and not on the chosen extension. Second, we observe that for
any principal bundle $p:\Cal G\to M$ over $M$, the lifted distribution $\widetilde
F\subset T\Cal G$ admits local frames that consist of projectable vector fields,
which implies that $\widetilde F$ is involutive, too. Hence the concept of partial
$V$-valued forms makes sense on the total space of $\Cal G$ and we also have an
exterior derivative for these forms.

\begin{definition}\label{def2.6.1}
Let $(p:\Cal G\to M,\om)$ be a partial Cartan geometry of type $(G,P)$ on
$(M,F)$. The \textit{curvature} $K\in\Om^2_{\widetilde F}(\Cal G,\frak g)$ of $\om$ is defined
by $$K(u)(X,Y)=d\om(u)(X,Y)+[\om(u)(X),\om(u)(Y)]$$ for $u\in\Cal G$ and $X,Y\in
T_u\Cal G$.
\end{definition}

From this definition, it follows exactly as for usual Cartan geometries that $K$ is
$P$-equivariant, i.e.\ $(r^g)^*K=\Ad(g^{-1})\o K$ for any $g\in P$, and horizontal,
i.e.\ it vanishes upon insertion of one tangent vector which is vertical for $p:\Cal
G\to M$. Thus one can equivalently encode $K$ via the \textit{curvature function}
$\ka:\Cal G\to L(\La^2\frak g/\frak p,\frak g)$, which is characterized by
\begin{equation}
\ka(u)(X+\frak p,Y+\frak p)=K(\om(u)^{-1}(X),\om(u)^{-1}(Y)). 
\end{equation}
This is well defined since $K$ is horizontal and it is $P$-equivariant by
construction. Observe that we can project the values of $\ka(u)$ to $\frak g/\frak p$
to obtain a function $\tau:\Cal G\to L(\La^2\frak g/\frak p,\frak g/\frak p)$. This
is also $P$-equivariant and hence defines an element of $\Om^2_F(M,F)$, which is
called the \textit{torsion} of the partial Cartan geometry $(p:\Cal G\to M,\om)$.

\medskip

For the case of AHS-structures, this readily implies that the standard notion of
normality continues to make sense in our setting. There is a linear map
$\partial^*:L(\La^2\frak g/\frak p,\frak g)\to L(\frak g/\frak p)$ which is defined
using the fact that $\frak g/\frak p$ and $\frak g_1$ are dual representations of $P$
via the Killing form of $\frak g$, see \cite{AHS2}. For $\ph\in L(\La^2\frak g/\frak
p,\frak g)$ and $X\in\frak g/\frak p$ one defines
\begin{equation}\label{codiff}
  (\partial^*\ph)(X):=\sum_i[Z^i,\ph(X_i,X)],
\end{equation}
where the $Z^i$ form a basis for $\frak g_1$ and the $X_i$ are the elements of the
dual basis of $\frak g/\frak p$.

\begin{definition}\label{def2.6.2}
  Let $(G,P)$ be an AHS-pair as in Definition \ref{def2.1}. Then a partial Cartan
  geometry $(p:\Cal G\to M,\om)$ of type $(G,P)$ is called \textit{normal} if it
  curvature function $\ka$ has the property that $\partial^*\o \ka=0$.
\end{definition}

The reason why $\partial^*$ gives rise to a good normalization condition are its
algebraic properties. First, since $\frak g/\frak p$ and $\frak g_1$ are dual
representations of $P$, the definition readily implies that $\partial^*$ is
$P$-equivariant. In particular, its kernel is a $P$-invariant subspace of
$L(\La^2\frak g/\frak p,\frak g)$. Second, the definition immediately implies that if
$\ph$ has values in $\frak g_i\subset\frak g$, then $\partial^*\ph$ has values in
$\frak g_{i+1}$. (This property is most naturally phrased as $\partial^*$ preserving
homogeneities, but we will not use this concept here.)

On the other hand, we can identify $\frak g/\frak p$ as a vector space with the
abelian Lie subalgebra $\frak g_{-1}\subset\frak g$. Consequently, there is a Lie
algebra cohomology differential $\partial:L(\frak g_{-1},\frak g)\to L(\La^2\frak
g_{-1},\frak g)$. Since $\frak g_{-1}$ is abelian, the formula for this differential
is just $\partial\ps(X,Y)=[X,\ps(Y)]-[Y,\ps(X)]$. Observe that if $\ps$ has values in
$\frak g_i$, then $\partial\ps$ has values in $\frak g_{i-1}$. In particular, we get
a restriction $\partial:L(\frak g_{-1},\frak g_0)\to L(\La^2\frak g_{-1},\frak
g_{-1})$. Viewing $\frak g_0$ as a Lie subalgebra of $\frak{gl}(\frak g_{-1})$ via
the adjoint action this is exactly the Spencer differential, which plays a crucial
role in the theory of $G_0$-structures.

The relevant algebraic properties then follow from the fact that $\partial$ and
$\partial^*$ are adjoint with respect to an appropriate inner product, see
\cite{Ochiai} or Section 3.3 of \cite{book} for a much more general version. In
addition, one needs information on the Lie algebra cohomology of $\frak g_{-1}$ with
coefficients in $\frak g$, which are obtained as a very special case of Kostant's
version of the Bott-Borel-Weil theorem from \cite{Kostant}, see Proposition 3.3.7 of
\cite{book} (taking into account that we excluded some pairs according to Definition
\ref{def2.1}). We will need the following facts that can be found in these
references.

\begin{prop}\label{prop2.6}
  Let $(G,P)$ be an AHS-pair as in Definition \ref{def2.1} with Lie algebras
  $\frak p\subset\frak g$.
  
  (1) For any $i=-1,0,1$, $\ker(\partial^*)$ and $\im(\partial)$ are complementary
  subspaces of $L(\La^2\frak g/\frak p,\frak g_i)$.

  (2) If $\ps\in L(\frak g_{-1},\frak g_0)$ satisfies $\partial\ps=0$, then there is
  a unique element $Z\in \frak g_1$ such that $\ps(X)=[Z,X]$ for any $X\in\frak
  g_{-1}$.

  (3) $\partial:L(\frak g_{-1},\frak g_1)\to L(\La^2\frak g_{-1},\frak g_0)$ is
  injective. 
\end{prop}

\subsection{Existence of normal partial Cartan geometries}\label{2.7}
In order to prove that normal partial Cartan geometries are an equivalent description
of partial AHS-structures, we first prove that any partial AHS-structure is induced
by some normal Cartan geometry.

\begin{thm}\label{thm2.7}
  Let $(G,P)$ be an AHS-pair as in Definition \ref{def2.1}, $G_0\subset P$ the
  structure group of the corresponding AHS-structure, and consider a partial
$G_0$-structure $(p_0:\Cal G_0\to M,\th)$ over $(M,F)$. Then there is a normal
partial Cartan geometry $(p:\Cal G\to M,\om)$, which has $(p_0:\Cal G_0\to M,\th)$ as
its underlying $G_0$-structure.
\end{thm}
\begin{proof}
We use a construction similar to the one in \cite{AHS2} respectively the one in
\cite{Cap-Schichl} (in the special case we consider here). We want to point out in
particular that all the technical verifications we need are the same for partial
Cartan geometries as for usual Cartan geometries, so we keep them rather short.

\textbf{Step 1: Construct the space $\Cal G$}: For a point $u_0\in\Cal G_0$, we
consider $\th(u_0):\widetilde{F}_{u_0}\to\frak g_{-1}$ and linear maps
$\ph:\widetilde{F}_{u_0}\to \frak g_0$, which extend $\th(u_0)$ to a linear map
$\widetilde{F}_{u_0}\to\frak g_{-1}\oplus\frak g_0$ requiring two more
conditions. First, we require that for the fundamental vector field $\ze_A$ generated
by $A\in\frak g_0$, we get $\phi(\ze_A(u_0))=A$. This in particular implies that
$\th(u_0)\oplus\ph$ is surjective and hence defines a linear isomorphism
$\widetilde{F}_{u_0}\to\frak g_{-1}\oplus\frak g_0$. Second, for $\xi,\eta\in
\widetilde{F}_{u_0}\Cal G_0$, we consider
\begin{equation}\label{tors}
d\th(u_0)(\xi,\eta)+\ph(\xi)(\th(u_0)(\eta))-\phi(\eta)(\th(u_0)(\xi)).
\end{equation}
If $\xi=\ze_A(u_0)$ for some $A\in\frak g_0$, then this reduces to
$d\th(u_0)(\ze_A,\eta)+[A,\th(u_0)(\eta)]$. This coincides with
$(\Cal L_{\ze_A}\th-\ad(A)\o\th)(u_0)(\eta)$ and hence vanishes by equivariancy of
$\th$. Thus there is a map $\tau_\phi:\La^2\frak g_{-1}\to\frak g_{-1}$ such that the
expression in \eqref{tors} equals $\tau_\phi(\th(u_0)(\xi),\th(u_0)(\eta))$, and the
second condition we require is that $\partial^*(\tau_\phi)=0$. For fixed $u_0$, we
can view the isomorphisms $\th(u_0)\oplus\ph$ for which both conditions are satisfied
as an affine subspace of the fiber of the vector bundle
$L(\widetilde{F},\frak g_{-1}\oplus\frak g_0)\to\Cal G_0$ over $u_0$. The union of
all these subspaces is easily seen to define a smooth submanifold of this total
space, which provides us which projections $q:\Cal G\to\Cal G_0$ and $p:\Cal G\to
M$. In what follows we will denote a typical element of $\Cal G$ as
$u=\th(u_0)\oplus\phi$.

\textbf{Step 2: Construct a $P$-action on $\Cal G$}: We can construct a natural right
action of $P$ on $L(\widetilde{F},\frak g_{-1}\oplus\frak g_0)$. On the one hand, we
get a natural representation of $P$ on $\frak g_{-1}\oplus\frak g_0$ by identifying
this space with $\frak g/\frak g_1$ and observing that $\frak g_1$ is invariant under
the adjoint action of $P$. With a slight abuse of notation we also denote this
representation by $\Adb$.  On the other hand, as noted in \S \ref{2.1}, $P/P_+\cong
G_0$ and for $g\in P$, we put $g_0:=gP_+\in G_0$. This justifies our abuse of notation
since viewing $\frak g_{-1}$ as $\frak g/\frak p$, for $g\in P$, the map $\Adb(g)$
descends to $\frak g_{-1}$ and since $P_+$ acts trivially on $\frak g/\frak p$, it
coincides with $\Adb(g_0)$ there.

Now for a linear map $\psi:\widetilde{F}_{u_0}\to \frak g_{-1}\oplus\frak g_0$, we define
$\psi\cdot g:\widetilde{F}_{u_0\cdot g_0}\to\frak g_{-1}\oplus\frak g_0$ by
\begin{equation}\label{act-def}
  \psi\cdot g:=\Adb(g^{-1})\o \psi\o Tr^{g_0^{-1}}.
\end{equation}
Since $\widetilde{F}$ is invariant under the principal right action, this clearly
defines a right action and we claim that the subset $\Cal G$ is $P$-invariant. First
observe that equivariancy of $\th$ reads as $\th\o Tr^{g_0^{-1}}=\Adb(g_0)\o
\th$. Together with our above observation, this readily implies that for
$u=\th(u_0)\oplus\phi \in\Cal G$ the $\frak g_{-1}$-component of $u\cdot g$ coincides
with $\th(u_0\cdot g_0)$. For fundamental vector fields, one has the equivariancy
property $Tr^{g^{-1}_0}(\ze_A(u_0\cdot g_0))=\ze_{\Ad(g_0)(A)}(u_0)$. Since $P_+$
acts trivially on $\frak p/\frak g_1$, we conclude as before that on
$\{0\}\oplus\frak g_0$, the map $\Adb(g^{-1})$ coincides with $\Ad(g_0^{-1})$ and
hence $u\cdot g$ reproduces the generators of fundamental vector fields.

To verify the last property for $u\cdot g$, recall from \S \ref{2.1} that $P$ is the
semi-direct product of $G_0$ and $P_+$, so we can write $g=g_0\exp(Z)$ for a unique
element $Z\in\frak g_1$, see also Theorem 3.1.3 of \cite{book}. Hence
$\Adb(g^{-1})=\Adb(\exp(-Z))\o\Adb(g_0^{-1})$ and since $\ad(Z)^2=0$ we conclude that
for $u=\th(u_0)\oplus\ph\in\Cal G$, the $\frak g_0$-component of $u\cdot g$ maps
$\xi\in \widetilde{F}_{u_0\cdot g_0}$ to
$$
\Ad(g_0^{-1})(\phi(Tr^{g_0^{-1}}(\xi))) -[Z,\Ad(g_0^{-1})(\th(u_0)(
  Tr^{g_0^{-1}}(\xi)))].
$$
Now if we consider \eqref{tors} for $u\cdot g$, then the terms involving $Z$ do not
give any contribution, since for $X,Y\in\frak g_{-1}$, we get
$0=[Z,[X,Y]]=[[Z,X],Y]+[X,[Z,Y]]$. Using this and that equivariancy of $\th$ implies
equivariancy of $d\th$, a short computation shows that $\tau_{u\cdot g}=g^{-1}_0\cdot
\tau_u$, where the dot denotes the canonical action of $G_0$ on $L(\La^2\frak
g_{-1},\frak g_{-1})$ induced by $\Adb$. But then equivariancy of $\partial^*$
implies that $\partial^*(\tau_{u\cdot g})=0$, which completes the proof that $\Cal G$
is invariant under the $P$-action we have constructed.

\textbf{Step 3: Show that $\Cal G\to M$ is a principal $P$-bundle}: We next show that
the $P$-action on $\Cal G$ is free and transitive on the fibers of $p:\Cal G\to
M$. To prove freeness, assume that $u\cdot g=u$ for some $u\in\Cal G$. Then of course
we must have $g_0=e$, so $g=\exp(Z)$ for some $Z\in\frak g_1$ and $\ad(-Z)$ has to 
vanish on $\frak g_{-1}$. But it is well known that $[Z,\frak g_{-1}]=\{0\}$ implies
$Z=0$, see Proposition 3.1.2 of \cite{book}. Next, we know that $G_0$ acts
transitively on the fibers of $\Cal G_0\to M$, so we can prove transitivity of the
action on the fibers by showing that $P$ acts transitively on the fibers of $q:\Cal
G\to \Cal G_0$. If $q(u)=q(\tilde u)=u_0$, we see that $\tilde u-u$ has to have
values in $\frak g_0$ and has to vanish on vertical vector fields. Hence there is a
linear map $f:\frak g_{-1}\to\frak g_0$ such that $\tilde
u(\xi)=u(\xi)+f(\th(u_0)(\xi))$ for each $\xi\in T_{u_0}\Cal G_0$. But looking at
equation \eqref{tors} we readily see that $\tau_{\tilde u}=\tau_u+\partial f$. Since
$u,\tilde u\in\Cal G$, this implies that $\partial^*\partial f=0$. By part (1) of
Proposition \ref{prop2.6}, this implies $\partial f=0$ which by part (2) of that
Proposition implies that $f=\ad(Z)$ for some $Z\in\frak g_1$. Hence we obtain $\tilde
u=u\cdot\exp(-Z)$ and transitivity of the action on the fibers is proved.

To prove that $p:\Cal G\to M$ is a $P$-principal bundle, it thus remains to construct
local smooth sections of $L(\widetilde{F},\frak g_{-1}\oplus\frak g_0)$ which have
values in $\Cal G$. Indeed, if $\si:U\to\Cal G$ is such a section, then $(x,g)\mapsto
\si(x)\cdot g$ defines a diffeomorphism $U\x P\to p^{-1}(U)$, whose inverse is a
principal bundle chart. To do this, we choose a principal connection $\ga$ on $\Cal
G_0$ and a local smooth section $\si_0:U\to\Cal G_0$ of $p_0:\Cal G_0\to M$. Then
$\hat\si(x):=\th(\si_0(x))\oplus\ga(\si_0(x))$ is an element of
$L(\widetilde{F},\frak g_{-1}\oplus\frak g_0)$ over $\si_0(x)\in\Cal G_0$. Plugging
this into \eqref{tors}, we conclude that $x\mapsto \tau_{\hat\si(x)}$ defines a
smooth map $U\to L(\La^2\frak g_{-1},\frak g_{-1})$. By part (1) of Proposition
\ref{prop2.6}, the target space splits as
$\ker(\partial^*)\oplus\im(\partial)$. Choosing left inverse of the surjective linear
map $\partial:L(\frak g_{-1},\frak g_0)\to \im(\partial)$ we obtain a smooth map
$\psi:U\to L(\frak g_{-1},\frak g_0)$ such that
$\tau_{\hat\si(x)}-\partial(\psi(x))\in\ker(\partial^*)$ for any $x\in U$.  But
\eqref{tors} immediately implies that this equals $\tau_{\si(x)}$, where
$$
\si(x):=\th(\si_0(x))\oplus \left(\ga(\si_0(x))-\psi(x)\o\th(\si_0(x))\right). 
$$
Since this is evidently smooth, we have constructed a local smooth section of $\Cal
G$ as required. 

\textbf{Step 4: Construct a normal Cartan connection}: From our construction of the
principal bundle $p:\Cal G\to M$, we conclude that there is a canonical section
$$
\hat\th=\hat\th_{-1}\oplus\hat\th_0 \in\Ga(L(\widetilde F,\frak g_{-1}\oplus\frak
g_0)\to\Cal G):
$$
Since a point $u\in\Cal G$ with $q(u)=u_0\in\Cal G_0$ is a linear isomorphism
$T_{u_0}\Cal G_0\supset\widetilde{F}_{u_0}\to\frak g_{-1}\oplus\frak g_0$ we can define
$\hat\th(u)(\xi):=u(T_uq(\xi))$ for $\xi\in T_u\Cal G$. This is obviously smooth and
the definition of the $P$-action on $\Cal G$ is exactly chosen in such a way that
$(r^g)^*\hat\th=\Adb(g^{-1})\o\hat\th$ so $\hat\th$ is equivariant for the
$P$-representation $\Adb$ on $\frak g_{-1}\oplus\frak g_0\cong \frak g/\frak g_1$.

The final step now is to extend, for a fixed point $u\in\Cal G$, the map $\hat\th(u)$
to a linear isomorphism $T_u\Cal G\to\frak g$ in such a way that the normalization
condition is satisfied. So given $u$, we want to define $\al:T_u\Cal
G\supset\widetilde{F}_u\to\frak g_1$ such that $\al(\ze_A(u))=0$ for $A\in\frak g_0$,
and $\al(\ze_Z(u))=Z$ for $Z\in\frak g_1$. In addition, similarly to \eqref{tors}, we
consider 
\begin{equation}\label{curv}
  d\hat\th_0(u)(\xi,\eta)+[\hat\th_0(u)(\xi),\hat\th_0(u)(\eta)]+
  [\al(\xi),\hat\th_{-1}(u)(\eta)]+[\hat\th_{-1}(u)(\xi),\al(\eta)]\in\frak g_0  
\end{equation}
for $\xi,\eta\in T_u\Cal G$. Inserting $\xi=\ze_B(u)$ for $B\in\frak g_0\oplus\frak
g_1$, we see that this reduces to the $\frak g_0$-component of $(\Cal
L_{\ze_B}\hat\th+\adb(B)\o\hat\th)(\eta)$, which vanishes by equivariancy of
$\hat\th$, As before, this shows that there is a map $\rho_\al\in L(\La^2\frak
g_{-1},\frak g_0)$ such that \eqref{curv} equals
$\rho_\al(\hat\th_{-1}(u)(\xi),\hat\th_{-1}(u)(\eta))$. We claim, there there is a
unique map $\al$ as above for which $\partial^*\rho_\al=0$. Indeed, by part (1) of
Proposition \ref{prop2.6}, $L(\La^2\frak g_{-1},\frak
g_0)=\ker(\partial^*)\oplus\im(\partial)$, so there is a linear map $\psi:\frak
g_{-1}\to\frak g_1$ such that $\partial^*(\rho_\al+\partial\psi)=0$. As before,
$\rho_\al+\partial Z=\rho_\be$ where $\be(\xi)=\al(\xi)+\psi(\hat\th_{-1}(\xi))$,
which proves existence.

To prove uniqueness, we observe that two maps $\al$ and $\be$ compatible with
fundamental vector fields are related by $\be(\xi)=\al(\xi)+\psi(\th_{-1}(u)(\xi))$
for a linear map $\psi:\frak g_{-1}\to\frak g_1$. Then
$\rho_{\be}=\rho_{\al}+\partial\psi$ and if
$\partial^*\rho_\al=\partial^*\rho_\be=0$, we conclude that
$\partial^*\partial\psi=0$. By part (1) of Proposition \ref{prop2.6} this implies
$\partial\psi=0$, which by part (3) of the same Proposition implies
$\psi=0$. Defining $\om_i=\hat\th_i$ for $i=-1,0$ and letting
$\om_1(u):\widetilde{F}_u\to \frak g_1$ be the map we just constructed, we obtain a
linear map $\om(u):\widetilde{F}_u\to\frak g$ which is surjective and hence a linear
isomorphism by construction. It also follows by construction that $\om(u)$ reproduces
the generators of fundamental vector fields.

Now for $g\in P$ consider $\Ad(g)\o \om(u\cdot g)\o T_ur^g:\widetilde{F}_u\to \frak
g$. By equivariancy of $\hat\th$, its component in $\frak g_{-1}\oplus\frak g_0$
coincides with $\hat\th(u)$ and equivariancy of fundamental vector fields implies
that it reproduces their generators. Similarly as above, one computes that the
associated map $\rho$ is related to the one for $\om(u)$ by the adjoint action of
$g_0=gP_+$ and hence lies in the kernel of $\partial^*$. Thus we conclude that
$\Ad(g)\o \om(u\cdot g)\o T_ur^g=\om(u)$ and hence $\om$ is $P$-equivariant. But we
also conclude that to deduce smoothness of $\om$ it suffices to deduce smoothness
along the images of local sections of $\Cal G\to M$. But along such sections the
construction above can be carried out depending on a point starting from a principal
connection on $\Cal G\to\Cal G_0$. This implies smooth dependence as above, so we
have verified that $\om\in\Ga(L(\widetilde{F},\frak g))$ is a partial Cartan
connection.

By construction the component $\om_{-1}$ coincides with the pullback of
$\hat\th_{-1}$ and hence with the pullback of $\th$, while $\om_0$ coincides with the
pullback of $\hat\th_0$. On the one hand, this implies that the partial $\Cal
G_0$-structure underlying $(p:\Cal G\to M,\om)$ is $(p_0:\Cal G_0\to
M,\th)$. Plugging into the defining equation of curvature and using compatibility of
the exterior derivative with pullbacks, we deduce information on the components
$\ka_i$ of the curvature function $\ka$, which have values in $L(\La^2\frak
g_{-1},\frak g_i)$. Namely, $\ka_{-1}$ is the pullback of \eqref{tors}, while $\ka_0$
is the pullback of \eqref{curv}, so both these components have values in
$\ker(\partial^*)$. But the whole subspace $L(\La^2\frak g_{-1},\frak g_1)$ is
contained in $\ker(\partial^*)$, we have proved that $\om$ is normal and this
completes the proof.
\end{proof}

\begin{remark}\label{rem2.7}
  (1) The construction in the proof of Theorem \ref{thm2.7} is easily seen to be
  functorial and hence together with the construction of the underlying partial
  $G_0$-structure defines an equivalence of categories. We will not need this result,
  however, since we will prove a stronger uniqueness result below, so we just sketch
  the argument.

  Suppose that we have given a morphism $\Ph$ between partial $G_0$-structures
  $(\Cal G_0^i\to M^i,F^i,\th^i)$ for $i=1,2$. Then for any $u_0\in\Cal G_0^1$,
  $T_{u_0}\Ph$ restricts to a linear isomorphism
  $\widetilde{F}^1_{u_0}\to\widetilde{F}^2_{\Ph(u_u)}$. Construct
  $(\Cal G^i\to M^i,\om^i)$ for $i=1,2$ as in the proof, so a point $u\in\Cal G^1$
  over $u_0\in\Cal G^1_0$ is a linear isomorphism
  $\widetilde{F}^1_{u_0}\to\frak g_{-1}\oplus\frak g_0$. This can be composed with
  $(T_{u_0}\Ph|_{\widetilde{F}^1_{u_0}})^{-1}$ to obtain a linear isomorphism defined
  on $\widetilde{F}^2_{\Ph(u_0)}$. Using that $\Ph^*\th^2=\th^1$ and naturality of
  the construction, one verifies that this isomorphism actually lies in $\Cal G^2$,
  let's call it $\tilde\Ph(u)$.

  Hence one obtains a lift $\tilde\Ph:\Cal G^1\to\Cal G^2$ of $\Phi$, which by
  construction has the property that $\tilde\Ph^*\hat\th^2=\hat\th^1$. The
  construction also implies that $\tilde\Phi$ is $P$-equivariant and thus a morphism
  of principal bundles. This also implies that for any $u\in\Cal G^1$, the tangent
  map $T_u\tilde\Ph$ restricts to a linear isomorphism
  $\widetilde{F}^1_u\to\widetilde{F}^2_{\tilde\Ph(u)}$. This implies that
  $\om^2(\tilde\Ph(u))\o T_u\tilde\Ph$ defines a linear isomorphism
  $\widetilde{F}^1_u\to\frak g$ whose component in $\frak g_{-1}\oplus\frak g_0$
  coincides with $\tilde\Ph^*\hat\th^2=\hat\th^1(u)$ by construction. Naturality of
  the normalization condition then implies that $\om^2(\tilde\Ph(u))\o T_u\tilde\Ph$
  satisfies the condition that characterizes $\om^1(u)$, and this implies that
  $\tilde\Ph$ is a morphism of partial Cartan geometries.

  \smallskip

  (2) There is a conceptual interpretation of the construction parallel to the case
  of usual AHS-structures. The linear maps $T_{u_9}\Cal G_0\to\frak g_0$ that occur
  in the first step can be viewed a values of partial principal connection forms in
  $u_0$. There is a correspondence between partial principal connections and partial
  linear connections on $F$, i.e.\ operators $\Ga(F)\times\Ga(F)\to\Ga(F)$ which
  satisfy properties analogous to the defining properties of a linear
  connection. Involutivity of $F$ implies that such a partial linear connection has a
  well defined torsion, which is a section of $L(\La^2F,F)$. The map $\partial^*$
  induces a bundle map $L(\La^2F,F)\to L(F,L(F,F))$ which is analogous to a
  normalization condition on the torsion in the theory of $G_0$-structures. In the
  picture of partial principal connections, this torsion is induced by equation
  \eqref{tors} and so the condition $\partial^*\o\tau_\ph=0$ singles out the values
  of partial principal connections with normal torsion. So one can view the bundle
  $\Cal G$ as obtained from attaching to each point $u_0\in\Cal G_0$ the values of
  the normal partial (principal) connections compatible with the partial
  $G_0$-structure in $u_0$. The pullback of a normal compatible partial connection
  along a morphism of partial $G_0$-structures is a normal compatible partial
  connection, so we obtain a lift to the extended bundle. To make this into a
  principal bundle and construct canonical Cartan connections, the specific algebraic
  setup for the AHS pairs is needed.

  \smallskip

  (3) In \S \ref{2.1} we excluded a specific one-grading of
  $\mathfrak{sl}(n+1,\Bbb R)$ for $n\geq 2$ from our discussion. For this grading,
  $G=SL(n+1,\Bbb R)$ and $P$ the stabilizer of a line in $\Bbb R^{n+1}$ is a possible
  choice of groups. For this choice, $G/P$ is the projective space $\Bbb RP^n$, and
  it is well known that this leads to a geometric structure which behaves very
  similarly to the AHS-structures we discuss. These structures are called (classical)
  projective structures and they are not determined by an underlying
  $G_0$-structure. Given an $n$-dimensional manifold $M$, a projective structure on
  $M$ is defined as an equivalence class of torsion-free linear connections on $TM$
  with two connections being equivalent iff they have the same geodesics up to
  parametrization. Alternatively, this property can be easily characterized by the
  fact that the $\binom{1}{2}$-tensor field that represents the difference of the two
  connections has a special form. It can then be shown that a projective structure
  can be equivalently encoded as a Cartan geometry of type $(G,P)$, which provides
  the analogy to the AHS-structures, see Section 4.1.5 of \cite{book}.

  Now we have the necessary concepts at hand to explain that also this case admits a
  generalization to the partial setting. Given a manifold $M$ of dimension $n+k$ and
  an involutive distribution $F\subset TM$ of rank $n$, we can consider partial
  linear connections $\nabla:\Ga(F)\x\Ga(F)\to\Ga(F)$ as in (2) above and require
  that the partial torsion defined there vanishes identically. Both versions of the
  definition of projective equivalence admit an obvious analog in this setting, and
  one defines a partial projective structure via an equivalence class of torsion-free
  partial connections. As detailed in \cite{Michal-thesis}, the construction of a
  canonical Cartan connection extends to the partial setting. In terms of the proof
  of Theorem \ref{thm2.7}, one starts from the full linear frame bundle $\Cal G_0$ of
  $F$ and attaches to $u_0\in\Cal G_0$ the connection forms of the partial principal
  connection forms corresponding to the partial linear connections in the class.

  The main motivation for these developments in \cite{Michal-thesis} is that
  canonical partial projective structures arise in unexpected situations. The main
  example discussed there is a contact manifold $(M,H)$ endowed with an involutive
  Legendrean distribution $F\subset H$. One then shows that a choice of contact form
  gives rise to a partial linear connection on $F$, and different choices lead to
  projectively equivalent partial connections. Hence one obtains a canonical partial
  projective structure on $(M,F)$, which turns out to always be locally flat. The
  construction of partial BGG sequences we obtain in \S \ref{3} below thus leads to
  differential complexes consisting of (higher order) operators that are intrinsic to
  the geometry $(M,H,F)$.

  The motivation for these developments came from the theory of Legendrean contact
  structures, which belong to the class of parabolic geometries. Such a structure is
  given by a decomposition of $H$ into a direct sum of two (not necessarily
  involutive) Legendrean distributions $E$ and $F$, see Section 4.2.3 of
  \cite{book}. In that case there is the general concept of Weyl connections, which
  produces a distinguished connection for any choice of contact form. Moreover, the
  complexes above are obtained using the machinery of relative BGG sequences
  developed in \cite{Rel-BGG2}. In \cite{Michal-thesis} it is shown that these
  constructions generalize in the way described above to either a single involutive
  Legendrean distribution $F$ or a general Legendrean distributions $F$ together with
  a choice of a (not necessarily Legendrean) complement $E$ in $H$.
\end{remark}

\subsection{Uniqueness}\label{2.8}
The construction of a canonical Cartan connection that we used in the proof of
Theorem \ref{thm2.7} has the advantage that it is geometric in nature. However, there
are very different constructions, for example parallel to the one in Section 3.1 of
\cite{book}, where one starts the construction by making choices and then just
exploits the algebraic properties of the normalization condition. Finally, there are
constructions for Cartan geometries like Fefferman-type constructions or extension
functors which directly produce normal Cartan geometries. Therefore, it is important
to prove uniqueness results for normal Cartan geometries which do not depend on
specific constructions. In our setting, this result looks as follows.

\begin{thm}\label{thm2.8}
  Let $(G,P)$ be an AHS-pair as in Definition \ref{def2.1}.  Suppose that $(p^i:\Cal G^i\to
  M^i,\om^i)$ are normal partial Cartan geometries over $(M^i,F^i)$ for $i=1,2$ with
  underlying partial $G_0$-structures $(p^i_0:\Cal G^i_0\to M^i,\th^i)$. Then any
  morphism $\Phi$ between the underlying partial $G_0$-structures uniquely lifts to a
  morphism of the partial Cartan geometries. In particular, any two normal partial
  Cartan geometries that have the same underlying partial $G_0$-structure are
  isomorphic via an isomorphism inducing the identity on the underlying partial
  $G_0$-structure.
\end{thm}
\begin{proof}
Start with $(M,F)$ and a principal $P$-bundle $\Cal G\to M$ and suppose that
$\om,\tilde\om\in\Ga(L(\widetilde{F},\frak g))$ are partial Cartan connections on
$\Cal G$, which induce the same underlying partial $G_0$-structure $(p_0:\Cal G_0\to
M,\th)$. The latter condition means that for each $u\in\Cal G$ and the components in
$\frak g_{-1}$ we get $\om_{-1}(u)=\tilde\om_{-1}(u)$. We write $\th(u)$ for this
component, since it corresponds to the pullback of $\th$ along the projection $q:\Cal
G\to\Cal G_0$. Now $\tilde\om-\om\in\Ga(L(\widetilde{F},\frak g))$ vanishes on the
vertical subspace and has values in $\frak g_0\oplus\frak g_1$, so there is a smooth
map $\psi:\Cal G\to L(\frak g_{-1},\frak g_0\oplus\frak g_1)$ such that
\begin{equation}\label{psidef}
  \tilde\om(u)(\xi)=\om(u)(\xi)+\psi(u)(\th(u)(\xi))
\end{equation}
for each $\xi\in T_u\Cal G$. Moreover, equivariancy of $\om$ and $\tilde\om$ easily
implies that $\psi(u\cdot g)=g^{-1}\cdot\psi(u)$ where in the right hand side we use
the action of $P$ coming from the identification of the target space with $L(\frak
g/\frak p,\frak p)$.

Now denote by $\ka$ and $\tilde\ka$ the curvature functions of the two Cartan
connections and by $\ka_i$ and $\tilde\ka_i$ their components in $\frak g_i$ for
$i=-1,0,1$. Then \eqref{psidef} easily implies that for the lowest components, we get
$$
\tilde\ka_{-1}(u)=\ka_{-1}(u)+\partial(\psi_0(u)),  
$$
where $\psi_0$ denotes component of $\psi$ in $L(\frak g_{-1},\frak g_0)$. Since both
$\tilde\ka$ and $\ka$ have values in $\ker(\partial^*)$, we conclude that
$0=\partial^*(\partial(\psi_0(u)))$. By Proposition \ref{prop2.6}, this implies that
$\partial(\psi_0(u))=0$ and hence $\psi_0(u)=\partial(Z(u))$ for a unique element
$Z(u)\in\frak g_1$.

Next, we claim that we can choose $Z(u)$ depending smoothly on $u$ and such that for
$g\in P$, we obtain $Z(u\cdot g)=\Ad(g_0^{-1})(Z(u))$, where, as before,
$g_0=gP_+$. For an open subset $U\subset M$ for which there is a smooth section
$\si:U\to\Cal G$, the map $x\mapsto\psi_0(\si(x))$ is smooth and by uniqueness
$Z(\si(x))$ has to be obtained from $\psi_0(\si(x))$ by composition with the inverse
of the linear isomorphism $\partial:\frak g_1\to\ker(\partial^*)\subset L(\frak
g_{-1},\frak g_0)$. Thus $x\mapsto Z(\si(x))$ is smooth and we can uniquely extend
this to a smooth map $p^{-1}(U)\to\frak g_1$ by putting $Z(\si(x)\cdot
g):=\Ad(g^{-1})(Z(\si(x)))$. Of course, this map then satisfies $Z(u\cdot
g)=\Ad(g^{-1})(Z(u))$ for any $u\in p^{-1}(U)$ and $g\in P$ and since $P_+$ acts
trivially on $\frak g_1$, this coincides with the required equivariancy
condition. Finally, since $\psi_0(\si(x)\cdot g)=g_0^{-1}\cdot\psi_0(\si(x))$,
$G_0$-equivariancy of $\partial$ implies that $\psi_0(u)=\partial Z(u)$ for all $u\in
p^{-1}(U)$. Doing this for the elements of an open covering of $M$, the corresponding
local maps can be glued to a global map with the required properties using (the
pullback of) a partition of unity on $M$.

Now $Z(u\cdot g)=\Ad(g^{-1})(Z(u))$ implies that $\exp(Z(u\cdot
g))=g^{-1}\exp(Z(u))g$ which shows that $u\mapsto u\cdot\exp(Z(u))$ defines a fiber
bundle map $\Xi:\Cal G\to\Cal G$ which covers the identity on the underlying partial
$G_0$-structure and is $P$-equivariant, so $\Xi$ is a morphism of principal
bundles. Thus $\Xi^*\tilde\om$ is a Cartan connection on $\Cal G$ and its curvature
function is given by $\tilde\ka\o\Xi$, so it also has values in
$\ker(\partial^*)$. For $u\in\Cal G$ and $\xi\in T_u\Cal G$, we get
$\Xi^*\tilde\om(u)(\xi)=\tilde\om(u\cdot\exp(Z(u))(T_u\Xi(\xi))$. Now the definition
of $\Xi$ readily implies that $T_u\Xi(\xi)=T_ur^{\exp(Z(u))}(\xi)+\eta$, where
$\eta$ lies in the vertical bundle of $q:\Cal G\to\Cal G_0$ and hence
$\tilde\om(\eta)\in\frak g_1$. Equivariancy of $\tilde\om$ thus implies that, modulo
elements of $\frak g_1$, we get
$$
\Xi^*\tilde\om(u)(\xi)=\Ad(\exp(-Z(u))(\tilde\om(u)(\xi))=\tilde\om(u)(\xi)-
   [Z(u),\th(u)(\xi)].  
$$
But the last term coincides with $-\psi_0(u)(\th(u)(\xi))$, so $\Xi^*\tilde\om$ and
$\om$ have the same $\frak g_0$-component.

So let us assume that in the above procedure we know that the components of
$\tilde\om$ and $\om$ in $\frak g_{-1}\oplus\frak g_0$ agree. Then the map $\psi$ in
\eqref{psidef} has values in $\frak g_1$ and from the definitions it follows readily
that $\tilde\ka_{-1}=\ka_{-1}$, while
$\tilde\ka_0(u)=\ka_0(u)+\partial(\psi(u))$. Normality of both Cartan connections
then implies $\partial^*( \partial(\psi(u)))=0$, which by Proposition \ref{prop2.6}
implies $\partial(\psi(u))=0$ and therefore $\psi(u)=0$. Hence we conclude that
$\Xi^*\tilde\om=\om$ so any two normal Cartan connections on $\Cal G$ that induce the
same underlying $G_0$-structure are related by an automorphism of $\Cal G$ which
induces the identity on this underlying structure.

To get the general statement on existence, we assume that $\Phi:\Cal G^1_0\to\Cal
G^2_0$ is a morphism of $G_0$-principal bundles with base map $\ph$ such that
$\Phi^*\th^2=\th^1$. Using that $\exp:\frak g_1\to P_+$ is a diffeomorphism, one
shows that the $P_+$-principal bundles $q^i:\Cal G^i\to\Cal G_0^i$ are trivial for
$i=1,2$, and there even are $G_0$-equivariant smooth sections $\si_i:\Cal
G_0^i\to\Cal G^i$ for $i=1,2$, see Proposition 5.1.1 of \cite{book}. Now there is a
principal bundle homomorphism $\tilde\Phi:\Cal G^1\to\Cal G^2$ characterized by
$\tilde\Ph\o\si_1=\si_1\o\Phi$, which of course lifts $\Phi:\Cal G^1_0\to\Cal
G^2_0$. Then $\tilde\Phi^*\om^2$ defines a partial Cartan connection on $\Cal G^1$
which is normal as the pullback of a normal connection. Hence from above, we know
that there is a morphism $\Xi:\Cal G^1\to\Cal G^1$ that covers the identity on $\Cal
G^1_0$ such that $\Xi^*\tilde\Ph^*\om^2=\om^1$ and hence $\tilde\Ph\o\Xi$ is a
morphism of Cartan geometries that lifts $\Ph$.

To prove uniqueness, we observe that any other homomorphism $\Cal G^1\to\Cal G^2$
that lifts $\Phi$ is of the form $\hat\Phi(u)=\tilde\Ph(u)\cdot\exp(Z(u))$ for a
smooth function $Z:\Cal G^1\to \frak g_1$ such that $Z(u\cdot g)=\Ad(g^{-1})(Z(u))$.
If $\tilde\Phi^*\om^2=\om^1$ then as above we conclude that for $u\in\Cal G$ and
$\xi\in T_u\Cal G$ the $\frak g_0$-component of $(\hat\Phi^*\om^2)(x)\xi$ coincides
with the the one of $\om^1(u)(\xi)-[Z(u),\th(u)(\xi)]$. Thus
$\hat\Phi^*\om^2(u)=\om^1(u)$ is only possible of $[Z(u),X]=0$ for any $X\in\frak
g_{-1}$, which implies $Z(u)=0$ and hence $\hat\Phi=\tilde\Phi$.
\end{proof}

\section{The partial BGG machinery}\label{3}

Having the setup of partial Cartan geometries corresponding to AHS-pairs $(G,P)$ at
hand, we can introduce BGG sequences following ideas from standard BGG sequences for
AHS-structures on the one hand an the relative BGG machinery introduced in
\cite{Rel-BGG2} on the other hand. Basically, one can use all the algebraic results
related to standard BGG sequences in the AHS case, while the differential operator
part is best done following the relative machinery.

\subsection{The partial fundamental derivative}\label{3.1}
Let $(G,P)$ be an AHS-pair as in Definition \ref{def2.1} and let
$(p:\Cal G\to M,\om)$ be a partial Cartan geometry of type $(G,P)$ on $(M,F)$. Then
we define the \textit{partial adjoint tractor bundle} $\Cal A_FM$ as the associated
bundle $\Cal G\x_P\frak g$. Sections of $\Cal A_FM$ are then in bijective
correspondence with smooth functions $f:\Cal G\to\frak g$ such that
$f(u\cdot g)=\Ad(g^{-1})(f(u))$ for any $u\in\Cal G$ and $g\in P$. But by definition,
$\om(u)$ defines a linear isomorphism $\widetilde F_u\to\frak g$, so given $f$ as
above, there is a unique element $\xi(u)\in \widetilde F_u$ such that
$\om(u)(\xi(u))=f(u)$. This defines a bijection between $C^\infty(\Cal G,\frak g)$
and $\Ga(\widetilde F)$ and using equivariancy of $\om$, the equivariancy condition
on $f$ translates to $\xi(u\cdot g)=T_ur^g(\xi(u))$ i.e.\ to $(r^g)^*\xi=\xi$. Thus
we see that sections of $\Cal A_FM$ are in bijective correspondence with
$P$-invariant sections of $\widetilde F\to\Cal G$.

This shows that we get counterparts to all the well-known operations on adjoint
tractor bundles in our setting:
\begin{itemize}
\item A bilinear bundle map $\{\ ,\ \}:\Cal A_FM\x\Cal A_FM\to\Cal A_FM$ induced by
  the Lie bracket on $\frak g$, which makes $\Cal A_FM$ into a bundle of Lie
  algebras.
\item A filtration of $\Cal A_FM$ by smooth subbundles of the form
  $\Cal A^1_FM\subset\Cal A^0_FM\subset\Cal A^{-1}_FM=\Cal A_FM$ induced by
  $\frak g_1\subset\frak p\subset\frak g$. This is compatible with $\{\ ,\ \}$ in the
  sense that (with obvious conventions) $\{\Cal A^i_FM,\Cal A^j_FM\}\subset\Cal
  A^{i+j}_FM$. Moreover, projecting invariant vector fields to $M$ gives rise to
  an isomorphism $\Cal A_FM/\Cal A^0_FM\to F$.
\item A Lie bracket $[\ ,\ ]:\Ga(\Cal A_FM)\x\Ga(\Cal A_FM)\to\Ga(\Cal A_FM)$ induced
  by the Lie bracket of vector fields. Viewing the map $\Cal A_FM\to F$ from above as
  having values in $TM$ and using it as an anchor, this makes $\Cal A_FM$ into a Lie
  algebroid over $M$.
\item By construction, the curvature function $\ka$ of $\om$ is equivariant, and
  hence it can be equivalently interpreted as a section of
  $\La^2F^*\otimes\Cal A_FM$, i.e.\ a partially defined two-form with values in
  $\Cal A_FM$.
\end{itemize}

Now assume that $\Bbb W$ is any representation of $P$. Then we get an associated
natural vector bundle $WM:=\Cal G\x_P\Bbb W\to M$ over $M$, and sections of $WM$ are
in bijective correspondence with smooth functions $f:\Cal G\to\Bbb W$ such that
$f(u\cdot g)=g^{-1}\cdot f(u)$. We can use any vector field on $\Cal G$ to
differentiate such a function, but the result will not be equivariant in general. But
if we take invariant vector fields and in particular the invariant sections of
$\widetilde F$ corresponding to sections of $\Cal A_FM$, we will end up with
equivariant functions. Thus, we get an operator $\Cal D:\Ga(\Cal
A_FM)\x\Ga(WM)\to\Ga(WM)$ which we denote by $(s,\si)\mapsto \Cal D_s\si$ in order to
emphasize the analogy to a (partial) covariant derivative. Indeed, the operator is
bilinear and linear over smooth functions in the first variable. In the second
variable, we have a Leibniz rule of the form $\Cal D_s(h\si)=(\Pi(s)\cdot h)\si+
h\Cal D_s\si$ for $h\in C^\infty(M,\Bbb R)$, where $\Pi$ denotes the tensorial
operator induced by the projection $\Cal A_FM\to F$. The construction also implies
immediately that this operation is compatible with all natural bundle maps that come
from $P$-equivariant maps between the inducing representations. Finally, since the
operation is tensorial in the first variable, we can also view it as $\Cal
D:\Ga(WM)\to \Ga(\Cal A^*_FM\otimes WM)$ via $(\Cal D\si)(s)=D_s(\si)$.

\subsection{Partial tractor connection and partial twisted exterior
  derivative}\label{3.2} Consider a representation $\Bbb V$ of $G$ that we can view
as a representation of $P$ by restriction. The we call the associated bundle
$VM:=\Cal G\x_P\Bbb V$ a \textit{partial tractor bundle}. Compared to a general
natural bundle, we now get an additional ingredient, namely a bilinear bundle map
$\Cal A_FM\x VM\to VM$ induced by the infinitesimal representation $\frak g\x\Bbb
V\to\Bbb V$. We will denote this bundle map and the induced tensorial operation on
sections by $(s,t)\mapsto s\bullet t$. This makes several algebraic operations
available which are the basis for the further developments. First, we get the
standard Lie algebra cohomology differential $\partial_{\frak g}:\La^k\frak
g^*\otimes\Bbb V\to\La^{k+1}\frak g^*\otimes\Bbb V$ for each $k=0,\dots,\dim(\frak
g)-1$. Viewing elements of $\La^*\frak g^*\otimes\Bbb V$ as multilinear alternating
maps, this is explicitly given by
\begin{equation}\label{partial-def}
\begin{aligned}
\partial_{\frak g}\al(A_0,\dots,A_k)&=\textstyle\sum_{i=0}^k(-1)^i
A_i\cdot\al(A_0,\dots,\widehat{A_i},\dots,A_n)\\
&+\textstyle\sum_{i<j}(-1)^{i+j}\al([A_i,A_j],A_0,\dots,\widehat{A_i},\dots,\widehat{A_j},\dots,A_k), 
\end{aligned}
\end{equation} 
where the dot in the first sum denote the action $\frak g\x \Bbb V\to\Bbb V$ and the
hats denote omission. Since this only uses the Lie bracket on $\frak g$ and the
action of $\frak g$ on $\Bbb V$ we immediately conclude that $\partial_{\frak g}$ is
$G$-equivariant and hence also $P$-equivariant. Consequently, it induces bundle maps
$\La^k\Cal A^*_FM\otimes VM\to\La^{k+1}\Cal A^*_FM\otimes VM$ and we also denote
these as well as the induced tensorial operations on sections by the same symbol
$\partial_{\frak g}$. Observe that we can view
$\ph\in\Ga(\La^k\Cal A^*_FM\otimes VM)$ as a $k$-linear alternating map
$\Ga(\Cal A^*_FM)^k\to\Ga(VM)$, which is linear over smooth functions in each
argument. In this picture, $\partial_{\frak g}\ph$ is given by a formula completely
analogous to \eqref{partial-def} but with the dot in the first sum replaced by
$\bullet$ and the bracket in the second sum replaced by $\{\ ,\ \}$.

There also is a natural differential operator acting between the same bundles. Given
$\ph\in\Ga(\La^k\Cal A^*_F\otimes VM)$, we can form $\Cal D\ph\in\Ga(\Cal A^*_FM\otimes
\La^k\Cal A^*_FM\otimes VM)$, thus defining a $(k+1)$-linear operator $\Ga(\Cal
A_F^*M)^{k+1}\to\Ga(VM)$ which is linear over smooth functions in each variable. We
can then make the result alternating by simply forming
\begin{equation}\label{Da-def}
(\Cal D^{\frak a}\ph)(s_0,\dots,s_k):=\textstyle\sum_{i=0}^k(-1)^i(\Cal
  D_{s_i}\ph)(s_0,\dots,\widehat{s_i},\dots,s_k) 
\end{equation}
for $s_j\in\Ga(\Cal A^*_FM)$. We can combine this with the Lie algebra cohomology
differential, by defining
$$ \tilde d^V_F:\Ga(\La^k\Cal A^*_FM\otimes VM)\to\Ga(\La^{k+1}\Cal A^*_FM\otimes VM)
\qquad \tilde d^V_F\ph:=\Cal D^{\frak a}\ph+\partial_{\frak g}\ph.
$$ The motivation for forming this combination is nicely visible by looking at the
case $k=0$, where we start with $t\in\Ga(VM)$ and then obtain $(\tilde d^V_Ft)(s)=\Cal
D_st+s\bullet t$. Now suppose that $s$ is a section of the subbundle $\Cal
A^0_FM\subset\Cal A_FM$. By definition this means that for the invariant vector field
$\xi\in\Ga(\widetilde F)$ corresponding to $s$ and a point $u\in\Cal G$, we get
$\om(\xi(u))\in\frak p$ and putting $\om(\xi(u))=A$, this means that
$\xi(u)=\ze_A(u)$. Denoting by $f:\Cal G\to\Bbb V$ the equivariant function
corresponding to $s$, equivariancy implies that $\xi(f)(u)=-A\cdot f(u)$, which
equals the value of the function corresponding to $-s\bullet t$ in $u$. Hence we
conclude that $\tilde d^V_Ft(s)=0$ for any $s\in\Ga(\Cal A^0_FM)$ and since $\Cal
A_FM/\Cal A^0_FM\cong F$, we conclude that $\tilde d^V_Ft$ actually is a section of
$F^*\otimes VM$. So we can view this as an operator $d^V_F:\Ga(VM)\to\Ga(F^*\otimes
VM)$ and the Leibniz rule for $\Cal D$ immediately implies that this defines a
partial connection, called the \textit{partial tractor connection} on the partial
tractor bundle $VM$.

It is rather easy to extend this to higher degrees. The bundle $\La^kF^*\otimes VM$
can be naturally viewed as the subbundle of $\La^k\Cal A^*_FM\otimes VM$, consisting
of those maps $(\Cal A_FM)^k\to VM$ that vanish upon insertion of one element of $\Cal
A^0_FM$. If $\phi$ is a section of this subbundle, then by the naturality discussed
in \S \ref{3.1} also $\Cal D_s\ph$ has this property for each $s\in\Ga(\Cal
AM)$. Using this, one can compute explicitly the behavior of $\tilde d^V_F\ph$ under
insertion of a section of $\Cal A^0_FM$ exactly as in the the first part of the proof
of Theorem 4.1 of \cite{Rel-BGG2} to obtain the following result. In what follows, we
use the notation $\Om^k_F(M,VM):=\Ga( \La^kF^*\otimes VM)$ to emphasize the point of
view of ``partially defined differential forms''.  
\begin{prop}\label{prop3.2}
Let $(G,P)$ be an AHS-pair and let $(p:\Cal G\to M,\om)$ be a partial Cartan geometry
of type $(G,P)$ on $(M,F)$. For a representation $\Bbb V$ of $G$, let $VM$ be the
corresponding partial tractor bundle. Then the operators $\tilde d^V_F$ restrict to
natural differential operators $d^V_F:\Om^k_F(M,VM)\to\Om^{k+1}_F(M,VM)$ for each
$k=0,\dots,n-1$. 
\end{prop}

\subsection{The algebraic Hodge theory}\label{3.3}
The second main ingredient needed for our version of the BGG construction is the
special case of B.\ Kostant's algebraic Hodge theory from \cite{Kostant} that is used
in the theory of AHS-structures. Since $\frak g_{-1}$ is a Lie subalgebra of
$\frak g$, we can restrict the infinitesimal action $\frak g\x \Bbb V\to\Bbb V$ from
\S \ref{3.2} to $\frak g_{-1}$. This makes $\Bbb V$ into a representation of the
(Abelian) Lie algebra $\frak g_{-1}$, so there also is a Lie algebra cohomology
differential
$\partial:=\partial_{\frak g_{-1}}:\La^k\frak g_{-1}^*\otimes\Bbb V\to\La^{k+1}\frak
g_{-1}\otimes\Bbb V$. This is defined by formula \eqref{partial-def}, but since
$\frak g_{-1}$ is Abelian, only the first sum in the formula survives. Compared to the
case of $\partial_{\frak g}$ discussed above, we have weaker equivariancy properties,
however. The adjoint action of elements of $G_0$ by definition maps $\frak g_{-1}$ to
itself, which easily implies that $\partial$ is $G_0$-equivariant. However, this is
not true for elements of $P$ and the standard way to define an action on
$\frak g_{-1}$ is via the identification with $\frak g/\frak p$. It turns out,
however, that this action does \textit{not} lead to $P$-equivariancy of the
infinitesimal representation and of $\partial$. Hence there is no immediate way to
convert $\partial$ into a geometric operation, and we'll discuss below how to deal
with this problem.

On the other hand, $\frak g_1\subset\frak g$ also is an Abelian Lie subalgebra of
$\frak g$ and by restriction $\Bbb V$ becomes a representation of $\frak g_1$. Dually
to Lie algebra cohomology, there is a standard complex computing Lie algebra homology
of $\frak g_1$ with coefficients in $\Bbb V$. The differential in this complex maps
$\La^k\frak g_1\otimes\Bbb V$ to $\La^{k-1}\frak g_1\otimes\Bbb V$. In the theory of
parabolic geometries, this is traditionally referred to as the \textit{Kostant
  codifferential} and denoted by $\partial^*$. Explicitly,
\begin{equation}\label{part*def}
\partial^*(Z_1\wedge\dots\wedge Z_k\otimes
v)=\sum_{i=1}^k(-1)^iZ_1\wedge\dots\wedge\widehat{Z_i}\wedge\dots \wedge Z_k\otimes
(Z_i\cdot v). 
\end{equation}
Now $\frak g_1$ actually \textit{is} invariant under the adjoint action of elements
of $P$ and then the definition readily implies that $\partial^*$ is
$P$-equivariant. As we has noted in \S \ref{2.4}, the Killing form of $\frak g$
induces a $P$-equivariant isomorphism between $\frak g_1$ and $(\frak g/\frak p)^*$,
so in particular $\frak g_1\cong \frak g_{-1}^*$ as a $G_0$-module. Thus $\partial^*$
acts between the same spaces as $\partial$ but in the opposite
direction. Indeed, Kostant's original definition of $\partial^*$ was as the adjoint
of $\partial$ with respect to a certain inner product, the algebraic
interpretation was noted only later (and provides a nice explanation for the better
equivariancy properties). The adjointness easily implies the following result, which
is known as the \textit{algebraic Hodge decomposition}, see \cite{Kostant} and
Theorem 3.3.1 of \cite{book} for a proof.

\begin{thm}\label{thm3.3}
  For any $k=0,\dots,n=\dim(\frak g_{-1})$ consider the $G_0$-equivariant map
  $\square=\square_k:=\partial^*\o\partial+\partial\o\partial^*$ acting on
  $\La^k\frak g_{-1}^*\otimes\Bbb V$. Then there is a $G_0$-invariant decomposition
  \begin{equation}\label{Hodge}
   \La^k\frak g_{-1}^*\otimes\Bbb V=\im(\partial)\oplus
   \ker(\square)\oplus \im(\partial^*),
  \end{equation}
  with the first two summands adding up to $\ker(\partial)$ and the last two summands
  adding up to $\ker(\partial^*)$.
\end{thm}

The $G_0$-invariant subspace $\ker(\square)$ can be identified with
$\ker(\partial)/\im(\partial)$ and thus with the Lie algebra cohomology $H^{k-1}(\frak
g_{-1},\Bbb V)$. However, we can also identify it with
$\ker(\partial^*)/\im(\partial^*)=H_k(\frak g_1,\Bbb V)$, which endows it naturally
with the structure of a $P$-module. It turns out that $P_+\subset P$ acts trivially
on this quotient. In \cite{Kostant}, B.\ Kostant analyzed the map $\square$ in more
detail (see also below) to obtain a complete description of the $G_0$-representation
$\ker(\square)$ in terms of highest weights. These weights are algorithmically
computable, so one can always determine the representations $H_k(\frak g_1,\Bbb V)$
explicitly.

\subsection{Converting to geometry}\label{3.4}
Parts of the developments of \S \ref{3.3} can be directly translated to geometric
objects on a partial Cartan geometry $(p:\Cal G\to M,\om)$ of type $(G,P)$. The
$P$-equivariant maps $\partial^*$ give rise to natural bundle maps
$\La^kF^*\otimes VM\to\La^{k-1}F^*\otimes VM$. We denote these bundle maps and the
induced tensorial operators on sections all by the same symbol $\partial^*$. In
particular, $\im(\partial^*)\subset\ker(\partial^*)\subset\La^kF^*\otimes VM$ are
natural subbundles and the quotient $\ker(\partial^*)/\im(\partial^*)$ can be
naturally identified with the associated bundle
$\Cal H_k^V:=\Cal G\x_PH_k(\frak g_1,\Bbb V)$. As we have noted above, $P_+$ acts
trivially on these homology spaces, so
$\Cal H_k^V=\Cal G_0\x_{G_0}H_k(\frak g_1,\Bbb V)$. Once the homology representations
are known, these bundles thus admit a direct interpretation in terms of the
underlying $G_0$-structure.

For the remaining ingredients from \S \ref{3.3} the translation to geometry is more
subtle and requires an additional ingredient. There is a so-called \textit{grading
  element} $E\in\frak g$ such that $[E,A]=jA$ for each $A\in\frak g_j$ with
$j=-1,0,1$. The definition of $E$ implies that it lies in the center of the subalgebra
$\frak g_0$ and in the cases we consider here, it actually spans this center. Lie
theory implies that $E$ acts diagonalizably on any representation $\Bbb V$ of $\frak
g$. Assuming that $\Bbb V$ is irreducible, the resulting eigenspaces can be used to
define a decomposition $\Bbb V=\oplus_{i=0}^N\Bbb V_i$ which has the property that
for $A\in\frak g_j$ and $v\in\Bbb V_i$, one obtains $A\cdot v\in\Bbb V_{i+j}$,
c.f.\ Section 2.5 of \cite{Rel-BGG1}. In particular, each of the subspaces $\Bbb V_i$
is $G_0$-invariant while $\Bbb V^i:=\oplus_{j\geq i}\Bbb V_j$ is even
$P$-invariant. Otherwise put, $\Bbb V$ admits a $P$-invariant filtration $\Bbb V=\Bbb
V^0\supset\Bbb V^1\supset\dots\supset\Bbb V^{N-1}\supset\Bbb V^N$.

Now this filtration of course induces a filtration of the partial tractor bundle $VM$
by the natural smooth subbundles $V^jM=\Cal G\x_P\Bbb V^j$. We can form the
\textit{associated graded vector bundle} $\gr(VM)=\oplus_{i=0}^N\gr_i(VM)$ where
$\gr_i(VM):=V^iM/V^{i+1}M$. By construction $\gr_i(VM)=\Cal G\x_P\gr_i(\Bbb V)$,
where $\gr_i(\Bbb V):=\Bbb V^i/\Bbb V^{i+1}$, so $\gr(VM)$ is the natural vector
bundle associated to $\gr(\Bbb V)=\oplus\gr_i(\Bbb V)$, the associated graded
representation to $\Bbb V$. Observe that by construction the action of $P_+$ maps
each $\Bbb V_i$ to $\Bbb V_{i+1}$, so the induced action on $\gr_i(\Bbb V)$ is
trivial, while as a $G_0$-module we get $\gr_i(\Bbb V)\cong\Bbb V_i$ and hence
$\gr(\Bbb V)\cong\Bbb V$. It is crucial to realize however, that this identification
of $G_0$ modules has no geometric counterpart. On the geometric side, we only have
the canonical quotient maps $V^iM\to\gr_i(VM)$ but there are no natural injective
bundle maps (in either direction) between $VM$ and $\gr(VM)$.

We can now view the $G_0$-equivariant maps $\partial$ and $\square$ as
$P$-equivariant maps on the spaces $\La^k\frak g_{-1}^*\otimes\gr(\Bbb V)$ and this
interpretation also makes sense for $\partial^*$. The explicit formulae immediately
imply that
\begin{equation}\label{homog}
  \begin{gathered}
    \partial(\La^k\frak g_{-1}^*\otimes\Bbb V_i)\subset \La^{k+1}\frak
    g_{-1}^*\otimes \Bbb V_{i-1} \\
    \partial^*(\La^k\frak g_{-1}^*\otimes \Bbb V_i)\subset \La^{k-1}\frak
    g_{-1}^*\otimes \Bbb V_{i+1} \\
    \square(\La^k\frak g_{-1}^*\otimes \Bbb V_i)\subset \La^k\frak
    g_{-1}^*\otimes \Bbb V_i
  \end{gathered}
\end{equation}
Hence we get induced bundle maps defined on the bundles $\La^kF^*\otimes\gr_i(VM)$
with analogous homogeneity properties. We will denote these bundle maps and the
induced tensorial operators on sections by the same symbols $\partial$, $\partial^*$
and $\square$. Notice that for $\partial^*$ this is an additional abuse of notation
since we denote the operations on the original bundles and on the associated graded by
the same symbol. This is justified by the homogeneity property in \eqref{homog}
which implies that they are compatible with the projection to the associated graded,
see below. Thus we obtain a geometric version of the Hodge decomposition from
\eqref{Hodge}.

\begin{cor}\label{cor3.4}
  For any $i=0,\dots,N$, there is a decomposition into a sum of natural subbundles
  \begin{equation}\label{Hodge-geom}
\La^kF^*\otimes\gr_i(VM)=\im(\partial)\oplus\ker(\square)\oplus\im(\partial^*) . 
  \end{equation}
  The bundle map $\square$ restricts to a linear isomorphism
  $\im(\partial^*)\to\im(\partial^*)$ and there is a universal polynomial $p^{\Bbb
    V}_{k,i}\in\Bbb R[x]$ (independent of $(M,F)$)such that the linear map $p^{\Bbb
    V}_{k,i}(\square)$ restricts on $\im(\partial^*)$ to
  $(\square|_{\im(\partial^*)})^{-1}$.
\end{cor}
\begin{proof}
  The decomposition \eqref{Hodge-geom} readily follows from Theorem \ref{thm3.3} and
  also the remaining statements are of algebraic origin. Since
  $\partial^*\o\partial^*=0$, we see that on $\im(\partial^*)$, $\square$ coincides
  with $\partial^*\o\partial$, so
  $\square(\im(\partial^*))\subset\im(\partial^*)$. Theorem \ref{thm3.3} also shows
  that $\ker(\partial^*)\cap\im(\partial^*)=\{0\}$ so $\square$ is injective on
  $\im(\partial^*)$. On the inducing representations, we are dealing with linear maps
  on finite dimensional vector spaces, so this implies that 
  $\square:\im(\partial^*)\to\im(\partial^*)$ is a linear isomorphism.

To prove the last part, we refer to Kostant's analysis of $\square$, see
\cite{Kostant} and Section 3.3 of \cite{book}. If $\frak g$ is complex and $\Bbb V$
is a complex representation of $\frak g$, the map $\square$ on $\La^kg_1\otimes\Bbb
V$ can be expressed in terms of Casimir operators. This then implies that
$\square$ acts diagonalizably with real eigenvalues on each $\La^k\frak
g_1\otimes\Bbb V_i$, hence this also holds for the subspace $\im(\partial^*)$. Linear
algebra then shows that, on this subspace, the inverse of $\square$ can be written as
a polynomial with real coefficients in $\square$. In the general case, we can
complexify $\frak g$ and $\Bbb V$ (as needed) and then view $\La^k\frak
g_1\otimes\Bbb V_i$ as a linear subspace in the analogous space obtained from the
complexifications. This subspace is invariant under $\square$, which implies that the
same polynomial works.
\end{proof}

The conceptual way to interpret the homogeneity properties is via \textit{total
  degree}, which is defined to be $k+i$ for elements of $\La^k\frak
g_1\otimes\gr_i(\Bbb V)$ (or of the corresponding bundle). Equation \eqref{homog}
then simply says that all three operations preserve the total degree. Similarly, we
will say that an element of $\Om^k_F(M,VM)$ is homogeneous of degree $\geq k+i$ if it
is a section of the subbundle $\La^kF^*\otimes V^iM$. If a section $\ph$ has this
property, then there is a well defined projection (which actually acts only on the
values of $\ph$) to $\gr_{k+i}(\ph)\in\Ga(\La^kF^*\otimes\gr_i(VM))$. This has
homogeneity $k+i$ and represents the lowest homogeneous component of $\ph$. In this
notation, the compatibility of $\partial^*$ with projection to the associated graded
reads as follows: If $\ph$ is homogeneous of degree $\geq k+i$, then so is
$\partial^*\ph$ and $\gr_{k+i}(\partial^*\ph)=\partial^*(\gr_{k+i}(\ph))$.

\subsection{The BGG machinery}\label{3.5}
Now we can quickly develop the BGG machinery in our setting following
\cite{Rel-BGG2}. The first step is analyzing the homogeneity properties of the
partial twisted exterior derivative, which suggests the definition of a geometric
version of the Kostant Laplacian.

\begin{lemma}\label{lem3.5}
(1) The twisted exterior derivative $d^V_F$ is compatible with (total) homogeneity
  and induces $\partial$ on the associated graded. More explicitly, if
  $\ph\in\Om^k_F(M,VM)$ has total homogeneity $\geq k+i$ that so has $d^V_F\ph\in
  \Om^{k+1}_F(M,VM)$ and for its projection to the associated graded, we get
  $\gr_{k+i}(d^V_F\ph)=\partial(\gr_{k+i}(\ph))$.

(2) The operator $\square^R:=\partial^*\o d^V_F+d^V_F\o \partial^*$ maps each
  $\Om^k_F(M,VM)$ to itself. It also preserves homogeneity and induces $\square$ on the
  associated graded.

  (3) Fix integers $k$ with $0\leq k\leq n$ and $i$ with $0\leq i\leq N$ and define
  $T_{k,i}:=p^{\Bbb V}_{k,i}(\partial^*\o d^V_F)$ where $p^{\Bbb V}_{k,i}$ is the
  polynomial from Corollary \ref{cor3.4}. Then $T_{k,i}$ maps
  $\Ga(\im(\partial^*))\subset\Om^k(M,VM)$ to itself and has the following property.
  If $\ps\in\Ga(\im(\partial^*))$ has homogeneity $\geq \ell:=k+i$, then the element
  $\tilde\ps:=d^V_F(T_{k,i}(\ps))\in\Om^{k+1}_F(M,VM)$ has homogeneity $\geq\ell$ and
  $\gr_\ell(\partial^*(\tilde\ps))=\gr_\ell(\ps)$.
\end{lemma}
\begin{proof}
(1) To prove the homogeneity statement, we assume $\ph\in\Ga(\La^kF^*\otimes V^iM)$
  for some $i$, and we have to show that $d^V_F\ph\in\Ga(\La^{k+1}F^*\otimes
  V^{i-1}M)$. To deal with $d^V_F$, we have to view $\ph$ as a section of $\La^k\Cal
  A^*M\otimes V^iM$ and use the definition of $\tilde d^V_F$. Naturality of $\Cal D$
  implies that for any $s\in\Ga(\Cal AM)$, $\Cal D_s\ph$ is a section of the
  subbundle $\La^kF^*\otimes V^iM$. Therefore each summand in the definition of $\Cal
  D^{\frak a}\ph$ in \eqref{Da-def} lies in $V^iM\subset V^{i-1}M$, so $D^{\frak a}$
  actually raises total homogeneity by $1$. On the other hand, the action of an
  element $A=A_{-1}+A_0+A_1\in\frak g$ maps $v\in\Bbb V_i$ to $\Bbb V_{i-1}\oplus\Bbb
  V_i\oplus\Bbb V_{i+1}$ and the component in $\Bbb V_{i-1}$ equals $A_{-1}\cdot
  v$. This readily implies that $\partial_{\frak g}(\ph)\in\Ga(\La^{k+1}F^*\otimes
  V^{i-1}M)$, so the homogeneity statement follows. But we also see that
  $\gr_{k+i}(\Cal D^{\frak a}\ph)=0$ and from the definitions and the above
  observation on the lowest component of $A\cdot v$, it immediately follows that
  $\gr_{k+i}(\partial_{\frak g}(\ph))=\partial(\gr_{k+i}(\ph))$.

   (2) We have noted in \S \ref{3.4} that
  $\partial^*:\Om^k_F(M,VM)\to\Om^{k-1}_F(M,VM)$ preserves homogeneity and induces
  $\partial^*$ on the associated graded. Hence both claims follow immediately from
  part (1).

  (3) Since $\partial^*\o d^V_F$ maps $\Ga(\im(\partial^*))$ to itself, so does
  $T_{k,i}$. Since $\partial^*$ preserves homogeneities, part (1) shows that $T_{k,i}$
  preserves being of homogeneity $\geq\ell$ for any $\ell$. If
  $\ps\in\Ga(\im(\partial^*))$ is homogeneous of degree $\geq\ell=k+i$ then so is
  $T_{k,i}(\ps)$ and this property is preserved by application of $d^V_F$ and then of
  $\partial^*$. Using part (1), we then compute that
  $$ 
  \gr_\ell(\partial^*(d^V_F(T_{k,i}(\ps))))=(\partial^*\o
  \partial)(\gr_\ell(T_{k,i}(\ps)))).
  $$ But for $\gr_\ell(T_{k,i}(\ps))$ we compute similarly step by step to conclude
  that this equals $p^{\Bbb V}_{k,i}(\partial^*\o\partial)(\gr_{\ell}(\psi))$. But on
  $\im(\partial^*)$, we have $\partial^*\o\partial=\square$ and then then results
  follows since $\square\o p^{\Bbb V}_{k,i}(\square)$ is the identity on
  $\Ga(\im(\partial^*))$ by Corollary \ref{cor3.4}.
\end{proof}

From this, the key elements of the BGG machinery follow quickly.

\begin{thm}\label{thm3.5}
  Consider the operator $\square^R:\Om^k_F(M,VM)\to\Om^k_F(M,VM)$ from Lemma
  \ref{lem3.5}.

  (1) $\ker(\square^R)=\Ga(\ker(\partial^*))\cap\ker(\partial^*\o d_F^V)$ and the bundle
  projection $$\ker(\partial^*)\to \ker(\partial^*)/\im(\partial^*)=\Cal H^V_k$$
  induces an isomorphism $\ker(\square^R)\to\Ga(\Cal H^V_k)$.

  (2) Viewed as an operator $S:\Ga(\Cal H^V)\to\Om^k_F(M,VM)$, the inverse of the
  isomorphism from (1) is a differential operator naturally associated to the given
  partial Cartan geometry. For $\al\in\Ga(\Cal H^V_k)$, the element
  $S(\al)\in\Om^k_F(M)$ is uniquely characterized by the facts that
  $\partial^*(S(\al))=0$, $\pi_H(S(\al))=\al$ and $\partial^*(d^V_F(S(\al)))=0$. Here
  $\pi_H$ denotes the quotient projection $\Ga(\ker(\partial^*))\to\Ga(\Cal H^V_k)$. 
\end{thm}
\begin{proof}
(1) We first claim that $\square^R$ is injective $\Ga(\im(\partial^*))$. Any section
  of this subbundle can be written as $\partial^*(\ps)$ for $\ps\in\Om^{k+1}_F(M,V)$
  and we have to prove that $\square^R(\partial^*(\ps))=0$ implies
  $\partial^*(\ps)=0$. To prove this, we assume inductively that for some $\ell\geq
  k$, $\partial^*(\ps)$ is homogeneous of degree $\geq \ell$. Since
  $\partial^*\o\partial^*=0$, we get
  $0=\square^R(\partial^*(\psi))=\partial^*(d^V_F(\partial^*(\ps)))$ and applying
  $\gr_\ell$ we conclude that
  $0=\partial^*(\partial(\gr_\ell(\partial^*(\ps))))$. Hence
  $\gr_\ell(\partial^*(\ps))$ is a section of the subbundle
  $\ker(\partial^*)\cap\im(\partial)\subset\La^kF^*\otimes\gr_{\ell-k}(VM)$, which is
  trivial by the algebraic Hodge decomposition \eqref{Hodge}. Since
  $\gr_\ell(\partial^*(\ps))=0$, we conclude that $ \partial^*(\ps)$ is homogeneous
  of degree $\geq\ell+1$. Since $\partial^*(\ps)$ automatically is homogeneous of
  degree $\geq k$ and the filtration of $\Bbb V$ has finite length, we conclude
  inductively that $\partial^*(\ps)=0$ as claimed.

  Now suppose that $\ph\in\Om^k_F(M,VM)$ satisfies $\square^R(\ph)=0$. Then
  $0=\partial^*(\square^R(\ph))$ and by definition, this equals
  $\square^R(\partial^*(\ph))$. From the injectivity statement above, we conclude
  $\partial^*(\ph)=0$ and hence $0=\square^R(\ph)=\partial^*(d^V_F(\ph))$. But this
  says that $\ph\in\Ga(\ker(\partial^*))\cap\ker(\partial^*\o d^V_F)$. Since the
  opposite inclusion is obvious, this implies the claimed description of
  $\ker(\square^R)$.

Next, we claim that the restriction of $\pi_H:\Ga(\ker(\partial^*))\to\Ga(\Cal
H^V_k)$ to $\ker(\square^R)$ is injective. Suppose that $\ph_1,\ph_2\in\Om^k_F(M,VM)$
satisfy $\partial^*(\ph_i)=0$ for $i=1,2$ and $\pi_H(\ph_1)=\pi_H(\ph_2)$. Then
$\ph_2=\ph_1+\partial^*(\ps)$ for some $\ps\in\Om^{k+1}_F(M,VM)$. But then if
$\square^R(\ph_i)=0$ for $i=1,2$ implies $\square^R(\partial^*(\ps))=0$ and hence
$\partial^*(\ps)=0$ and $\ph_1=\ph_2$ by the injectivity statement above.

Now for $\al\in\Ga(\Cal H^k_V)$, we can choose some representative
$\ph\in\Om^k_F(M,VM)$, i.e.\ $\partial^*(\ph)=0$ and $\pi_H(\ph)=\al$. Assume
inductively that $\square^R(\ph)=\partial^*(d^V_F(\ph))$ is homogeneous of degree
$\geq \ell=k+i$. Define
\begin{equation}\label{Si-def}
  S_i(\ph):=\ph-T_{k,i}(\partial^*(d^V_F(\ph))),
\end{equation}
where $T_{k,i}$ is the operator from part (3) of Corollary \ref{cor3.4}. Since
$T_{k,i}$ has values in $\Ga(\im(\partial^*))$, we see that $\partial^*(S_i(\ph))=0$
and $\pi_H(S_i(\ph))=\al$. Since $S_i(\ph)\in\Ga(\ker(\partial^*))$ we get
$\square^R(S_i(\ph))=\partial^*(d^v_F(S_i(\ph)))$. Part (3) of Corollary \ref{cor3.4}
then implies that $\partial^*(d^V_F(S_i(\ph)))$ is homogeneous of degree $\geq \ell$
and that $\gr_\ell(\partial^*(d^V_F(S_i(\ph))))=0$. Thus $\square^R(S_i(\ph))$ is
homogeneous of degree $\geq \ell+1$. Since $\square^R(\ph)$ is always homogeneous of
degree $\geq k$, we conclude inductively that we also obtain a representative $\ph$
for $\al$ such that $\square^R(\ph)=0$ and this completes the proof of (1).

\medskip

(2) We can view \eqref{Si-def} as defining a differential operator on all of
$\Ga(\ker(\partial^*))\subset\Om^k_F(M,VM)$ (noting that also there $\square^R$
coincides with $\partial^*\o d^V_F$). As above, this has values in
$\Ga(\ker(\partial^*))$ and it preserves total homogeneity as well as the subspace
$\Ga(\im(\partial^*))\subset\Ga(\ker(\partial^*))$. But by definition $S_i=\id-
T_{k,i}\o\square^R$ so for a section $\ps\in\Ga(\im(\partial^*))$, which is
homogeneous of degree $\geq k+i$, we conclude as above that $S_i(\psi)$ is
homogeneous of degree $\geq k+i+1$.

Now consider the operator $S:=S_N\o S_{N-1}\o\dots\o S_1\o S_0$ as mapping
$\Ga(\ker(\partial^*))$ to itself. This also preserves total homogeneity, and for
$\ps\in\Ga(\im(\partial^*))$ we immediately conclude from above that $S(\ps)$ is
homogeneous of degree $\geq k+N$ and hence $S(\ps)=0$. Hence $S$ descends to a
natural differential operator $\Ga(\Cal H^V_k)\to\Om^k_F(M,VM)$ and the last part of
the proof of (1) shows that this induces the inverse of the isomorphism
$\ker(\square^R)\to \Ga(\Cal H^V_k)$ constructed in (1). The characterization of this
inverse then readily follows from part (1).
\end{proof}

\begin{remark}\label{rem3.5}
As with the standard BGG machinery, the construction of the splitting operator in the
proof of Theorem \ref{thm3.5} is mainly useful for theoretical purposes. It does not
provide an effective description of the splitting operators or a reasonable way to
compute explicit formulae (except in extremely simple situations). The simplest way
to obtain explicit formulae is to determine $S(\al)$ by explicitly verifying the
characterizing conditions in part (2) of the theorem. (The best way to do this is
using the analog of the theory of Weyl structures as discussed in \cite{Weyl} and
Chapter 5 of \cite{book} in our setting.)

The splitting operator also is of much lower order than one would expect from the
construction in the proof. The first simplification arises from the fact that
$\ker(\square)$ is a completely reducible representation and each of the irreducible
components is contained in $\La^k\frak g_1\otimes\Bbb V_{i_0}$ for some fixed
$i_0$. Such a component corresponds to a subbundle in $\Cal H^V_k$, and for sections
of that subbundle, $S$ actually coincides with $S_N\o\dots\o S_{i_0+1}\o
S_{i_0}$. But still, the expected order of this is much higher than its actual
order. The expected order of the operator $S_i$ is $\deg(p^{\Bbb V}_{k,i})+1$, which
by Kostant's results roughly equals the number of $\frak g_0$-isotypical components
in the representation $\im(\partial^*)\subset\La^k\frak g_1\otimes\Bbb V_i$. So the
expected order of $S_N\o\dots\o S_{i_0+1}\o S_{i_0}$ roughly is the number of
isotypical components of $\im(\partial^*)\subset\La^k\frak g_1\otimes\Bbb
V^{i_0}$. But simple homogeneity considerations show that the actual order is
$\leq N-i_0$.
\end{remark}

\subsection{BGG sequences and BGG resolutions}\label{3.6}
The further developments mainly use formal properties that we have derived already
and hence are very closely parallel to the usual theory of BGG sequences. Thus we
keep the discussion rather brief. By part (2) of Theorem \ref{thm3.5}, the splitting
operators have the property that $d^V_F(S(\al))\in\Ga(\ker(\partial^*))$ for any
$\al\in\Ga(\Cal H^V_k)$, so these can be projected to the homology bundles to define
the BGG operators.

\begin{definition}\label{def3.6}
  Let $VM$ be a partial tractor bundle on $(M,F)$ induced by a representation $\Bbb
  V$ of $G$ and let $S$ denote the splitting operators from Theorem \ref{thm3.5}. For
  $\al\in\Ga(\Cal H^V_k)$, we define $D^V_F(\al):=\pi_H(d^V_F(S(\al))\in\Ga(\Cal
  H^V_{k+1})$. For each $k$, this defines a natural differential operator
  $D^V_F:\Ga(\Cal H^V_k)\to\Ga(\Cal H^V_{k+1})$ called the $k$th partial BGG operator
  associated to $\Bbb V$. The sequence $(\Ga(\Cal H^V_*),D^V_F)$ of operators is
  called the partial BGG sequence determined by $\Bbb V$.
\end{definition}

We do not go into a detailed discussion of examples of partial BGG sequences
here. The point about this is that the bundles that are contained a BGG sequence and
the orders of the BGG operators can be easily deduced from the corresponding
information for the usual BGG sequences associated to AHS structures. In the
description of the bundles, one basically only has to replace the tangent bundle by
$F$ and then apply algebraic constructions accordingly. The orders of the operators
are exactly the same as for standard BGG sequences. Indeed, an irreducible component
$W_1\subset H_k(\frak g_1,\Bbb V)$ has to be contained in $\La^k\frak g_1\otimes\Bbb
V_i$ for some fixed number $i_1$. Denoting by $i_2$ the corresponding number for an
irreducible component $W_2\subset H_{k+1}(\frak g_1,\Bbb V)$, the order of the
component of the BGG operator mapping sections of the bundle associated to $W_1$ to
sections of the one associated to $W_2$ simply is $i_2-i_1+1$. In particular, this
shows that for any type of AHS-structure, even the first operator in the BGG sequence
can have arbitrarily high order. Also the description of the principal parts of BGG
operators in terms of distinguished connections for a AHS structure admits and analog
using distinguished partial connections, but we do not go into this here.

\medskip

The comparison of the BGG sequence to the sequence $(\Om^*_F(M,VM),d^V_F)$ is
particularly interesting in the case that the latter sequence is a complex,
i.e.~$d^V_F\o d^V_F=0$. However, there is also some interesting information available
in general. In particular, in degree zero we obtain an analog of the concept of
normal solutions of the first BGG operator, see \cite{Leitner} and \cite{polynomiality}. 

\begin{thm}\label{thm3.6}
  For a partial tractor bundle $VM\to M$, consider the partial twisted exterior
  derivative $d^V_F$ and the induced partial BGG operators $D^V_F$. Then we have:

  (1) The linear subspace $\ker(d^V_F)\cap\Ga(\ker(\partial^*))\subset\Om^k_F(M,VM)$
  projects isomorphically onto a linear subspace of
  $\ker(D^V_F)\subset\Ga(\Cal H^V_k)$.  In degree zero, this means that the space of
  sections of $VM$ that are parallel for the partial tractor connection projects
  isomorphically onto a subspace of the kernel of the first partial BGG operator.

  (2) If $d^V_F\o d^V_F=0$ on $\Om^k_F(M,VM)$, then $D^V_F\o D^V_F=0$ on $\Ga(\Cal
  H^V_k)$. If this holds for all $k$, and hence both sequences are complexes, then
  their cohomologies are canonically isomorphic.
\end{thm}
\begin{proof}
  (1) Theorem \ref{thm3.5} (1) shows that
  $\Ga(\ker(\partial^*))\cap\ker(d^V_F)\subset\ker(\square^R)$ and hence $\pi_H$ is
  injective on this subspace. But for $\ph\in\ker(\square^R)$ we get
  $\ph=S(\pi_H(\ph))$ and thus $D^V_F(\pi_H(\ph))=\pi_H(d^V_F(\ph))=0$, which implies
  the claim. In degree zero, we have $\ker(\partial^*)=VM$ and $d^V_F$ is just the
  partial tractor connection, so the second statement follows.

  \smallskip
  
  (2) For $\al\in\Ga(\Cal H^V_k)$ consider
  $\ph:=d^V_F(S(\al))\in\Ga(\ker(\partial^*))$, which by definition satisfies
  $\pi_H(\ph)=D^V_F(\al)$. But by assumption $d^V_F(\ph)=0$ so part (1) shows that
  $\pi_H(\ph)=D^V_F(\al)\in\ker(D^V_F)$ and hence $D^V_F\o D^V_F=0$ on
  $\Ga(\Cal H^V_k)$. From the proof of part (1) we also see that $\ph=S(\pi_H(\ph))$, 
  which shows that $d^V_F(S(\al))=S(D^V_F(\al))$. Assuming that $d^V_F\o d^V_F=0$ in
  all degrees, we see that the BGG sequence is a complex and that the splitting
  operators satisfy $S\o D^V_F=d^V_F\o S$. Hence they define a chain map which in
  turn induces a map in cohomology.

  To prove the statement about cohomology, we first claim that for any element
  $\ph\in\Om^k_F(M,VM)$, there exists $\tilde\ph\in\Om^{k-1}_F(M,VM)$ such that
  $\ph+d^V_F(\tilde\ph)\in\Ga(\ker(\partial^*))$. Since $\ph$ is automatically has
  homogeneity $\geq k$, the same holds for $\partial^*(\ph)$. Assuming inductively
  that $\partial^*(\ph)$ is homogeneous of degree $\geq \ell=k+i$, we consider
  $d^V_FT_{k-1,i+1}(\partial^*(\ph))$ for the operator from part (3) of Corollary
  \ref{cor3.4}. By that Corollary, this has homogeneity $\geq \ell$ and applying
  $\gr_{\ell}$ we obtain $\gr_\ell(\partial^*(\ph))$. But this shows that
  $\partial^*(\ph-d^V_FT_{k-1,i+1}(\partial^*(\ph)))$ has homogeneity $\geq \ell+1$,
  and inductively the claim follows.  

  The claim readily implies that any cohomology class for $d^V_F$ has a
  representative which is a section of $\ker(\partial^*)$. If $\ph$ is such a
  representative, we get $\ph=S(\pi_H(\ph))$ and $D(\pi_H(\ph))=0$ by part (1). This
  shows that the map in cohomology induced by $S$ is surjective. To prove injectivity, 
  take $\al\in\ker(D^V_F)$ and assume that $S(\al)=d^V_F(\ps)$ for some
  $\ps\in\Om^{k-1}_F(M,VM)$. This implies $\partial^*(d^V_F(\psi))=0$ and by the
  claim and since $d^V_F\o d^V_F=0$, we may assume that
  $\ps\in\Ga(\ker(\partial^*))$. But then $\psi=S(\pi_H(\psi))$ by part (2) of
  Theorem \ref{thm3.5} and hence $\al=\pi_H(d^V_F(\psi))=D(\pi_H(\psi))$. This shows
  injectivity of the map in cohomology induced by $S$.
\end{proof}

\end{document}